\documentclass[
	a4paper,
	12pt,
	twoside
]{article}

	\usepackage[utf8]{inputenc}		
	\usepackage[T1]{fontenc}			
	\usepackage{lmodern}			
	\usepackage[ngerman,english]{babel}	
	\usepackage[
		autostyle=true,
		german=quotes
	]{csquotes}					
	\usepackage[
		left=2.5cm,
		right=2.5cm,
		top=2.5cm,
		bottom=2.5cm
	]{geometry}					
	\usepackage{fancyhdr}			
	\usepackage{titlesec}				
	\usepackage{amsmath}			
	\usepackage{amsfonts}			
	\usepackage{amssymb}			
	\usepackage{amsthm}			
	\usepackage{mathtools}			
	\usepackage{extarrows}			
	\usepackage[square,numbers]{natbib}
	\usepackage{aliascnt}				
	\usepackage[
		pdftoolbar=true,
		linktoc=all,
	]{hyperref}					

	\newcommand{\doi}[1]{\textsc{doi:} \href{https://doi.org/#1}{\texttt{#1}}}
	\newcommand{\arxiv}[2]{arXiv: \href{https://arxiv.org/abs/#1}{\texttt{#1 [#2]}}}
	\newcommand{\isbn}[1]{\textsc{isbn:} #1}
	\hypersetup{
		pdftitle={Kolmogorov bounds for decomposable random variables and subgraph counting by the Stein-Tikhomirov method},
		pdfsubject={},
		pdfauthor={Peter Eichelsbacher, Benedikt Rednoß},
		pdfkeywords={}
	}												

	\titleformat{\section}[hang]{\vspace{\baselineskip}\centering\uppercase}{\thesection}{1ex}{}{}
	\titleformat{\subsection}[hang]{\bfseries}{\thesubsection}{1ex}{}{}

	\newcounter{aussage}[section]
	
	\newcommand{\Umgebung}[2]{
		\newaliascnt{#1}{aussage}
		\newtheorem{#1}[#1]{#2}
		\newtheorem*{#1*}{#2}
		\aliascntresetthe{#1}
		\expandafter\providecommand\csname #1autorefname\endcsname{#2}	
	}
	
	\newtheoremstyle{nichtkursiv}{}{}{}{}{\bfseries}{}{ }{\thmname{#1}\thmnumber{ #2}\thmnote{ (#3)}}
	\newtheoremstyle{kursiv}{}{}{\itshape}{}{\bfseries}{}{ }{\thmname{#1}\thmnumber{ #2}\thmnote{ (#3)}}
	
	\theoremstyle{nichtkursiv}
	\Umgebung{defn}{Definition}
	\Umgebung{exmp}{Example}
	\Umgebung{rem}{Remark}
	
	\theoremstyle{kursiv}
	\Umgebung{theo}{Theorem}
	\Umgebung{prop}{Proposition}
	\Umgebung{lem}{Lemma}
	\Umgebung{cor}{Corollary}

	\newcommand{\N}{\mathbb{N}}									
	\newcommand{\R}{\mathbb{R}}									
	\newcommand{\C}{\mathbb{C}}									
	\newcommand{\dd}{\text{\normalfont d}}							
	\renewcommand{\L}{\mathcal{L}}								
	\newcommand{\Normal}{\mathcal{N}}							
	\newcommand{\A}{\mathcal{A}}								
	\newcommand{\HT}{\mathcal{H}}								
	\newcommand{\E}{\mathbb{E}}									
	\newcommand{\Var}{\text{\normalfont Var}}						
	\newcommand{\Cov}{\text{\normalfont Cov}}						
	\renewcommand{\binom}[2]{\begin{pmatrix} \,#1\, \\ \,#2\, \end{pmatrix}}	
	\newcommand{\Kol}{d_K}										
	\newcommand{\Wass}{d_1}									
	\newcommand{\G}{{G_0}}										

\begin{document}

\mbox{}
\begin{center}
{\large\bfseries\uppercase{Kolmogorov bounds for decomposable random variables and subgraph counting by the Stein-Tikhomirov method}}\\[2\baselineskip]
Peter Eichelsbacher\footnote{Ruhr-University Bochum, Faculty of Mathematics, IB 2/115, D-44780 Bochum, Germany. E-Mail: \href{mailto:peter.eichelsbacher@rub.de}{peter.eichelsbacher@rub.de}} and Benedikt Rednoß\footnote{Ruhr-University Bochum, Faculty of Mathematics, IB 2/95, D-44780 Bochum, Germany. E-Mail: \href{mailto:benedikt.rednoss@rub.de}{benedikt.rednoss@rub.de}}\\[2\baselineskip]
\thispagestyle{plain}
\end{center}
\begin{abstract}
	We derive normal approximation bounds in the Kolmogorov distance for random variables posessing decompositions of
	\citet*{BKR89}. We highlight the example of 
	normalized subgraph counts in the Erdös-Rényi random graph.
	We prove a bound by generalizing the argumentation of \citet{Ro17},
	who used the Stein-Tikhomirov method to prove a bound in the special case of normalized triangle counts.
	Our bounds match the best available Wasserstein-bounds.\\[0.5cm]
\end{abstract}


\section{Introduction}\label{sec:Intro}

In his work \cite{Ti80}, Tikhomirov combined elements of Stein's method with the theory of charac\-teristic functions to derive Kolmogorov bounds for the convergence rate in the central limit theorem for a normalized sum of a stationary sequence of random variables satisfying one of several weak dependency conditions.
The combination of elements of Stein's method with the theory of characteristic functions is sometimes called \emph{Stein-Tikhomirov method}.
\citet*{AMPS17} successfully used the Stein-Tikhomirov method to bound the convergence rate in contexts with non-Gaussian targets.
\citet*{Ro17} used the Stein-Tikhomirov method to bound the convergence rate in the Kolmogorov distance for normal approximation of normalized triangle counts in the Erdös-Rényi random graph.

Consider a random graph $G(n,p)$ on $n \in \N$ vertices.
Each edge between two vertices is included with probability $p \in [0,1]$ independently of all other edges.
$p$ may depend on $n$.
This model is called \emph{Erdös-Rényi random graph}, though it was first introduced by \citet*{Gi59}.
To a given graph $G$,
let $v_G$ denote the number of vertices of $G$,
$e_G$ the number of edges of $G$,
$H \subset G$ a subgraph of $G$.\\
Now let $\G$ be a graph with $e_\G > 0$.
It will be useful to introduce the quantity
\begin{align*}
	\Psi := \Psi_{\G} := \displaystyle \min \big\{ n^{v_H} \cdot p^{e_H} \,\big\vert\, H \subset \G,  e_H > 0 \big\},
\end{align*}
which was used by \citet*{BKR89}.\\
Let $W$ be the normalized number of subgraphs of $G(n,p)$ which are isomorphic to $\G$.
Let $Z$ denote a standard Gaussian random variable.\\
Throughout this paper, $C$ (with or without index) will denote a constant, which may vary from line to line.
The indices of $C$ will denote the quantities on which $C$ depends.

\citet*{Ru88} provided a necessary and sufficient condition for asymptotic normality of $W$:
$W$ converges in distribution to $Z$ if and only if
\begin{align*}
	\Psi \xlongrightarrow{n\to\infty} \infty
	\quad\text{and}\quad
	n^2 \cdot (1-p) \xlongrightarrow{n\to\infty} \infty,
\end{align*}
(see Theorem~2 in \cite{Ru88}).
Shortly thereafter, \citet*{BKR89} proved the following bound for the convergence rate in the Wasserstein distance $\Wass$,
using Stein's method
(see Theorem~2 in \cite{BKR89}):
\begin{align*}
	\Wass(\L(W), \Normal(0,1))
	\leq{}& C_\G \cdot
	\begin{dcases}
		\frac1{n\sqrt{1-p}}
		& \text{ if }p > \frac12,\\
		\frac1{\sqrt{\Psi}}
		& \text{ if }p \leq \frac12.
	\end{dcases}
\end{align*}
Here $\Wass(\cdot,\cdot)$ is defined in \eqref{eq:DistDef} with $\HT$ being the class of real-valued Lipschitz functions with Lipschitz constant less than or equal to 1.
Subsequently, it has been a long standing problem, whether the same rate could be achieved in the Kolmogorov distance $\Kol$. Here the class
of real-valued test function is chosen to be the indicators of half intervals $x \mapsto 1_{\{ x \leq w \}}$ with $w \in \R$. In general, the Kolmogorov distance
satisfies the bound
\begin{equation} \label{comparison}
\Kol (\L(W), \Normal(0,1) \leq \sqrt{\Wass(\L(W), \Normal(0,1))},
\end{equation}
which is in most cases a crude bound. Typically, it can be shown that the left hand side of  \eqref{comparison} is of the same order as the bound for $\Wass(\L(W), \Normal(0,1))$. However, the derivation of the latter requires a much more delicate argument. 
Throughout the years, there were several results for special cases:
For triangle counts with $p = \theta_n n^{-\alpha}$, $\alpha \in [0,1)$, $(\theta_n)_{n\in\N}$ bounded and $\liminf_{n \rightarrow \infty} \theta_n > 0$, \citet*{KRT17} proved a bound for the Kolmogorov distance, using Malliavin-Stein method (see Theorem~1.1 in \cite{KRT17}).
However, their rate in the Kolmogorov distance does not match the rate of \citet*{BKR89} for the Wasserstein distance.
For arbitrary subgraphs and fixed $p \in (0,1)$, \citet*{KRT17} as well as \citet*{FMN17} derived bounds in the Kolmogorov distance (see Theorem~1.2 in \cite{KRT17} and Example~37 in \cite{FMN17}) similar to the best known Wasserstein bound by \citet*{BKR89}.
For arbitrary $p$, but only for triangles as subgraphs to count, \citet*{Ro17} derived Kolmogorov bounds by means of the Stein-Tikhomirov method, which match the best known Wasserstein bounds (see Theorem~1.1 in \cite{Ro17}).
Finally, based on the work of \citet*{KRT17},
\citet*{PS20} proved for the general setting of an arbitrary subgraph and arbitrary $p$, that
\begin{align*}
	\Kol(\L(W), \Normal(0,1))
	\leq{}& C_\G \cdot
	\frac1{\sqrt{(1-p) \cdot \Psi}}
\end{align*}
(see Theorem~4.2 in \citet*{PS20}),
which is the same rate as the one proven by \citet*{BKR89} for the Wasserstein distance. From a point of view of probability theory,
the proof presented in \cite{PS20} is quite involved. First one has to develop the Malliavin approach for statistics of Bernoulli random variables,
or for Rademacher sequences, respectively. This was worked out in \cite{KRT17}, pages 1077--1087. Next the approach was combined
to the Stein method to derive general Berry-Ess\'een bounds, see Theorem 3.1 in \cite{KRT17}. The authors in \cite{PS20} modified this
general bound in Proposition 2.1. Next they derive new bounds for sums of discrete multiple integrals applying the well-known multiplication formula of Wiener
integrals, see Theorem 3.1 in \cite{PS20}. They show that the number of subgraphs can be represented as a finite sum of discrete multiple integrals. Finally they derive laborious bounds for certain corresponding norms of contractions to obtain the result. In contrast to the Malliavin-Stein
approach, \cite{BKR89} represented the number of subgraphs of $G(n,p)$ which are isomorphic to $\G$ as a decomposable random variable in the sense, that
it is a sum of weakly dependent variables. In section 3, the so called BKR-decomposition will be introduced. In the work of \citet*{R03}, the BKR-decomposition
was considered as well. The author managed to prove a Berry-Ess\'een theorem for BKR-decompositions, if the different components of this decomposition are assumed to be bounded, see Theorem 5.1 in \cite{R03}. But interesting enough, a direct application of this result -- in the case of counting subgraphs -- does not lead to an optimal result.

In our paper we generalize the argumentation of \citet*{Ro17} to provide a simpler proof for the result of \citet*{PS20}.
The starting point will be the quite natural BKR-decomposition which will be introduced in section 3.
In particular, we will show the following theorem, which is our main theorem.
\begin{theo}\label{maintheorem}
	For $p_0 \in (0,1)$ we have
	\begin{align*}
		\Kol(\L(W), \Normal(0,1))
		\leq{}& C_{\G,p_0} \cdot
		\begin{dcases}
			\frac1{n\sqrt{1-p}}
			& \text{ if }p > p_0,\\
			\frac1{\sqrt{\Psi}}
			& \text{ if }p \leq p_0.
		\end{dcases}
	\end{align*}
\end{theo}

\noindent
Our proof starts with the following considerations:\\
Let $W$ be an integrable random variable with characteristic function $\varphi$.
Following the Stein-Tikhomirov method we are interested in the quantity
\begin{flalign*}
	&&
	\varphi^\prime(t) + t \cdot \varphi(t)
	&= \E [ iW e^{itW} ] + t \cdot \E[ e^{itW} ]\\
	&&
	&= \E[ (iW+t) \cdot e^{itW} ]
	&\forall t &\in \R.
\end{flalign*}
To analyze this quantity, we want to construct a stochastic process $(H_t)_{t\in\R}$ with
\begin{flalign*}
	&&
	\varphi^\prime(t) + t \cdot \varphi(t) &= t^2 \cdot \E[ H_t \cdot e^{itW} ]
	&\forall t &\in \R
\end{flalign*}
and $H_t$ integrable.
Note that this does not necessarily require, that $t^2 \cdot H_t = iW+t$.
It is sufficient if there is a stochastic process $(Y_t)_{t\in\R}$ with
\begin{flalign*}
	&&
	\E[ Y_t \cdot e^{itW} ] &= 0 &
	\text{and}&&
	t^2 \cdot H_t &= iW+t + Y_t
	&\forall t &\in \R.
\end{flalign*}
Under these assumptions we find
\begin{align*}
	\varphi^\prime(t) + t \cdot \varphi(t)
	&= t^2 \cdot \E[H_t \cdot e^{itW}] \\
	&= t^2 \cdot \E[H_t] \cdot \E[ e^{itW}] + t^2 \cdot \E \big[ (H_t - \E[H_t]) (e^{itW} - \E[e^{itW}]) \big] \\
	&= t^2 \cdot \E[H_t] \cdot \varphi(t) + t^2 \cdot \Cov(H_t , e^{-itW}).
\end{align*}
In \autoref{sec:STM} we will use the Stein-Tikhomirov method to derive a bound for $\Kol(W,Z)$,
which only depends on upper bounds of $\vert \E[H_t] \vert$ and $\vert \Cov(H_t , e^{-itW}) \vert$.
Under the additional assumption, that there is a decomposition of $W$ similar to the decomposition used by \citet*{BKR89},
we will find general upper bounds for $\vert \E[H_t] \vert$ and $\vert \Cov(H_t , e^{-itW}) \vert$ in \autoref{sec:BKR}.
Finally, in \autoref{sec:MT} we apply the upper bounds from the previous section to the subgraph counting problem and will prove \autoref{maintheorem}.

\section{The Stein-Tikhomirov method}\label{sec:STM}
The combination of elements of Stein's method with the theory of characteristic functions is called Stein-Tikhomirov method.

\subsection{Stein's method}\label{sub:SM}
In his work \cite{St72}, Stein derived a bound for the Kolmogorov distance between a sum of dependent random variables and a Gaussian target.
The method he used was later generalized and is now known as \emph{Stein's method}.
The most common distances between two random variables $X$ and $Z$ can be written as
\begin{align} 
	d_\HT (X,Z) = \sup_{h \in \HT} \big\vert \E[h(X)] - \E[h(Z)] \big\vert
	\label{eq:DistDef}
\end{align}
with $\HT$ some set of test-functions.
Usually, $Z$ is the target random variable in an approximation process.
The main idea of Stein's method is to construct a so called \emph{Stein-operator} $\A$ on a set of functions $M$
to characterize the distribution of the target random variable $Z$ in the following sense:
The random variable $X$ has the same distribution as $Z$ if and only if $\E[\A f(X)]=0$ for all $f \in M$.
Now the right side of \eqref{eq:DistDef} shall be expressed by $\A$ and $X$,
so that the target random variable $Z$ does not directly appear anymore.
This leads to
\begin{align}
	d_\HT (X,Z) = \sup_{h \in \HT} \big\vert \E[\A f_h(X)] \big\vert
	\label{eq:SteinEqua}
\end{align}
with $f_h$ a solution of the so called \emph{Stein-equation} $\A f_h = h - \E[h(Z)]$. 
In many cases the right side of \eqref{eq:SteinEqua} is much easier to bound than the right side of \eqref{eq:DistDef}.

In case of a Gaussian target $Z$, the most common Stein-operator is
\begin{align}
	\A f (x) := f^\prime (x) - x \cdot f(x)
	\label{eq:SteinOpGauss}
\end{align}
for all absolutely continuous functions $f$, for which the expectation $\E[f^\prime(Z)]$ exists
(see e.g. Lemma~2.1 in \cite{Ro11}).
Throughout the years, Stein-operators for many other target distributions have been constructed.
For a more detailed introduction to Stein's method, see e.g. \cite{Ro11}.

\subsection{Application to characteristic functions}\label{sub:TM}
Following the argumentation of \citet*{Ti80},
the Stein-Tikhomirov method focuses on characteristic functions.
Let $\varphi_X$ and $\varphi_Z$ be the characteristic functions of the random variables $X$ and $Z$.
The main goal of the Stein-Tikhomirov method is to derive a bound for $\vert \varphi_X - \varphi_Z \vert$.
This bound is then used to derive a bound for a probability distance:
The smoothing-lemma by Berry and Ess\'een provides a mean to transfer a bound for $\vert \varphi_X - \varphi_Z \vert$ to a bound for the Kolmogorov distance:
\begin{theo}[Berry-Ess\'een smoothing-lemma (see e.g. \cite{Lo77})]\label{theo:BE}
	If at least one of the random variables $X$ and $Z$ has a density bounded by $b > 0$, then
	\begin{flalign*}
		\phantom{\forall T}&\phantom{>0}&
		\Kol(X,Z) &\leq \frac1\pi \int_{-T}^T \bigg\vert \frac{\varphi_X(t) - \varphi_Z(t)}{t} \bigg\vert \, \dd t + \frac{24b}{\pi T}
		&\forall\, T &> 0.
	\end{flalign*}
\end{theo}

\citet*{AMPS17} provided two theorems by which (under certain additional assumptions) a bound for $\vert \varphi_X - \varphi_Z \vert$ can also be transferred to a bound for a modified Wasserstein distance (see Theorems~1 and~2 in \cite{AMPS17}).
However, in our paper we will focus on the Kolmogorov distance.

To derive a bound for  $\vert \varphi_X - \varphi_Z \vert$ in the first place,
the main idea of the Stein-Tikhomirov method is to characterize the characteristic function of the target random variable $Z$
by the so called \emph{Tikhomirov-operator} $\L$,
analogously to the Stein-operator, in the following sense:
The random variable $X$ has the same distribution as $Z$ if and only if $\L \varphi_X (t) = 0$ for all $t \in \R$.
In the next step, a bound for $\vert \L \varphi_X \vert$ is constructed.
This bound is finally transferred to a bound for  $\vert \varphi_X - \varphi_Z \vert$.

\citet*{Ti80} worked with a Gaussian target $Z$ and therefore analyzed the following differential equation,
which can easily be expressed by a Tikhomirov-operator:
\begin{align*}
	\L \varphi_X (t) = t \cdot \varphi_X (t) + \varphi^\prime_X (t).
\end{align*}
This operator indeed characterizes the Gaussian distribution.
The construction of this operator can be motivated by the following considerations, which highlight the parallels to Stein's method:
Consider the functions $f_t: \R \rightarrow \C$, $x \mapsto e^{itx}$ with $t \in \R$.
Using the Stein-operator \eqref{eq:SteinOpGauss} on these functions directly leads to
\begin{align*}
	\E[\A f_t(X)] = t \cdot \varphi_X (t) + \varphi^\prime_X (t).
\end{align*}
\citet*{Ti80} analyzed this differential equation and constructed bounded functions $a$ and $b$, so that
\begin{align*}
	t \cdot \varphi_X (t) + \varphi_X^\prime (t)
	= t^2 \cdot \big( a(t) \cdot \varphi(t) + b(t) \big).
\end{align*}
He then proceeded to solve this differential equation to derive a bound for $\vert \varphi_X - \varphi_Z \vert$.
His calculation steps were later generalized by \citet*{Ro17}.
The important step can be summed up in the following lemma:
\begin{lem}\label{lem:ODEsol}
	Let $\varphi$ be the characteristic function of a random variable.
	Suppose, there are functions $a,b: \R \rightarrow \C$ and real values $A>0$, $B\geq 0$ so that
	$t \mapsto t^2 \cdot a(t)$ and $t \mapsto t^2 \cdot b(t)$ are continuous and
	\begin{flalign*}
		&&
		\varphi^\prime(t) + t \cdot \varphi(t) &= t^2 \cdot \big( a(t) \cdot \varphi(t) + b(t) \big)
		& \forall t &\in \R,\\
		&&
		\Vert a \Vert_\infty &\leq A,\\
		&&
		\Vert b \Vert_\infty &\leq B.
	\end{flalign*}
	Then
	\begin{align*}
		\Big\vert \varphi(t) - \exp\Big\{-\frac{t^2}2 \Big\} \Big\vert
		\leq{}&
		\frac{A}3 \cdot t^3 \cdot \exp\Big\{ -\frac{t^2}4 \Big\}
		+ 2Bt.
	\end{align*}
\end{lem}
\begin{proof}
	Since $\varphi$ is a characteristic function, we know that $\varphi(0)=1$.
	Therefore, there is a unique solution to the differential equation
	\begin{flalign*}
		&&
		\varphi^\prime(t) + t \cdot \varphi(t) &= t^2 \cdot \big( a(t) \cdot \varphi(t) + b(t) \big)
		& \forall t &\in \R.
	\end{flalign*}
	We get $\forall t \in \R$
	\begin{align*}
		\varphi(t)
		={}&
		\exp\Big\{ \int_0^t (-s + s^2 \cdot a(s) ) \,\dd s \Big\}
		+ \int_0^t \exp \Big\{ \int_u^t (-s + s^2 \cdot a(s) ) \,\dd s \Big\} \cdot u^2 \cdot b(u) \,\dd u\\
		={}&
		\exp\Big\{-\frac{t^2}2 + \int_0^t s^2 \cdot a(s) \,\dd s \Big\}
		+ \int_0^t \exp\Big\{ -\frac{t^2}2 + \frac{u^2}2 + \int_u^t s^2 \cdot a(s) \,\dd s \Big\} \cdot u^2 \cdot b(u) \,\dd u.
	\end{align*}
	The first summand can be expressed by
	\begin{align*}
		&\exp\Big\{-\frac{t^2}2 + \int_0^t s^2 \cdot a(s) \,\dd s \Big\} \\
		={}&
		\exp\Big\{-\frac{t^2}2 + \int_0^t s^2 \cdot a(s) \,\dd s \Big\} - \exp\Big\{-\frac{t^2}2 \Big\} + \exp\Big\{-\frac{t^2}2 \Big\} \\
		={}&
		\int_0^t s^2 \cdot a(s) \,\dd s \cdot \int_0^1 \exp\Big\{ -\frac{t^2}2 + u \cdot \int_0^t s^2 \cdot a(s) \,\dd s \Big\} \,\dd u
		+ \exp\Big\{-\frac{t^2}2 \Big\}.
	\end{align*}
	This leads to
	\begin{align*}
		\varphi(t) - \exp\Big\{-\frac{t^2}2 \Big\}
		={}&
		\int_0^t s^2 \cdot a(s) \,\dd s \cdot \int_0^1 \exp\Big\{ -\frac{t^2}2 + u \cdot \int_0^t s^2 \cdot a(s) \,\dd s \Big\} \,\dd u \\
		&+
		\int_0^t \exp\Big\{ -\frac{t^2}2 + \frac{u^2}2 + \int_u^t s^2 \cdot a(s) \,\dd s \Big\} \cdot u^2 \cdot b(u) \,\dd u.
	\end{align*}
	If $\vert t \vert \leq \frac1{2A}$ and $u \in [0,t]$ respectively $u \in [t,0]$, then
	\begin{align*}
		\Big\vert \int_u^t s^2 \cdot a(s) \,\dd s \Big\vert
		&\leq \int_{\vert u \vert}^{\vert t \vert} \frac{s}{2A} \cdot \vert a(s) \vert \,\dd s
		\leq \frac{t^2}4 - \frac{u^2}4,
	\end{align*}
	so that
	\begin{align*}
		&\Big\vert \varphi(t) - \exp\Big\{-\frac{t^2}2 \Big\} \Big\vert\\
		\leq{}&
		\frac{A}3 \cdot t^3 \cdot \int_0^1 \exp\Big\{ -\frac{t^2}2 + u \cdot \frac{t^2}4 \Big\} \,\dd u
		+
		\int_0^t \exp\Big\{ -\frac{t^2}2 + \frac{u^2}2 + \frac{t^2}4 - \frac{u^2}4 \Big\} \cdot u^2 \cdot B \,\dd u \\
		\leq{}&
		\frac{A}3 \cdot t^3 \cdot \exp\Big\{ -\frac{t^2}4 \Big\}
		+
		B \cdot \exp\Big\{ -\frac{t^2}4 \Big\} \cdot \int_0^t \exp\Big\{ \frac{u^2}4 \Big\} \cdot u^2 \,\dd u \\
		\leq{}&
		\frac{A}3 \cdot t^3 \cdot \exp\Big\{ -\frac{t^2}4 \Big\}
		+ 2Bt.
		\qedhere
	\end{align*}
\end{proof}
Note that if the functions $a$ and $b$ are not continuous, but if at least all the integrals exist,
then the result of \autoref{lem:ODEsol} is valid at least a.~e., which is sufficient for the next step.
Though, in the setting of our paper, it will be easy to show that $a$ and $b$ are continuous.
Since $t \mapsto \exp\{-\frac{t^2}2\}$, $t \in \R$, is the characteristic function of a Gaussian random variable,
by plugging the results of \autoref{lem:ODEsol} into the smoothing-lemma~(\autoref{theo:BE}), we finally get:
\begin{theo}[see Lemma~2.3 in \cite{Ro17}]\label{theo:Ti80}
	Let $\varphi$ be the characteristic function of a random variable.
	Suppose, there are functions $a,b: \R \rightarrow \C$ and real values $A>0$, $B\geq 0$ so that
	$t \mapsto t^2 \cdot a(t)$ and $t \mapsto t^2 \cdot b(t)$ are continuous and
	\begin{flalign*}
		&&
		\varphi^\prime(t) + t \cdot \varphi(t) &= t^2 \cdot \big( a(t) \cdot \varphi(t) + b(t) \big)
		& \forall t &\in \R,\\
		&&
		\Vert a \Vert_\infty &\leq A,\\
		&&
		\Vert b \Vert_\infty &\leq B.
	\end{flalign*}
	Then
	\begin{align*}
		\Kol(\L(W), \Normal(0,1))
		&\leq
		\Big( \frac{4}{3\sqrt\pi} + \frac{24\sqrt2}{\pi\sqrt{\pi}} \Big) \cdot A + \frac2\pi \cdot \frac{B}{A}.
	\end{align*}
\end{theo}

\citet*{AMPS17} describe, how to construct Tikhomirov-operators for a larger class of target distributions.
In particular, they describe ways of constructing such an operator without the need of an already existing Stein-operator.
They also show how to solve the resulting differential equations in the case,
that the upper bound for $\L\varphi_X$ does not include $\varphi_X$ itself.

\subsection{Results}\label{sub:RM}
Now, we can apply \autoref{theo:Ti80} to the context described at the end of the introduction:
\begin{cor}\label{theo:STMresult}
	Let $W$ be an integrable random variable with characteristic function $\varphi$.
	Suppose, there is a stochastic process $(H_t)_{t\in\R}$ so that there is $\forall t \in \R$
	\begin{align*}
		\varphi^\prime(t) + t \cdot \varphi(t) &= t^2 \cdot \E[ H_t \cdot e^{itW} ]
	\end{align*}
	and $H_t$ is integrable.
	If $t \mapsto t^2 \cdot \E[H_t]$ and $t \mapsto t^2 \cdot \Cov(H_t, e^{itW} )$ are continuous and
	if $\vert \E[ H_t ] \vert \leq A$ and $\vert \Cov( H_t  , e^{-itW}) \vert \leq B$ $\forall t \in \R$
	for some real values $A>0$, $B\geq 0$, then
	\begin{align*}
		\Kol(\L(W), \Normal(0,1))
		&\leq
		\Big( \frac4{3\sqrt\pi} + \frac{24 \sqrt2}{\pi\sqrt{\pi}} \Big) \cdot A + \frac2\pi \cdot \frac{B}{A}.
	\end{align*}
\end{cor}
In our application of \autoref{theo:STMresult} $\varphi$ will depend on a natural number $n \in \N$.
Therefore, $A$ and $B$ will also depend on $n$.
In most cases, they will converge to $0$ as $n \rightarrow \infty$.
We want the upper bound from \autoref{theo:STMresult} to also converge to $0$.
This is the case if $B = o(A)$.
Ideally, there is $B \approx A^2$.
It is also possible to use $1$ as bound for $\varphi$ in \autoref{lem:ODEsol}.
In this case, however, after some short computations, we finally get an upper bound for the Kolmogorov distance of order $\sqrt{A+B}$,
which is a weaker upper bound than the one in \autoref{theo:STMresult} in case of $B \approx A^2$.

\section{Barbour-Karo\'nski-Ruci\'nski-Decomposition}\label{sec:BKR}

Here, we introduce random variables posessing decompositions of \citet*{BKR89}, which are particularly useful in combinatorial structures, where
there is no natural ordering of the summands.  

\begin{defn}
Suppose that $W$ is a random variable decomposed in the following way: 
Let $J$ be a finite index set.
For $j \in J$ let $N_j$ be a subset of $J$.
Further, let
$\{ X_j \}_{j \in J}$,
$\{ W_j \}_{j \in J}$,
$\{ Z_j \}_{j \in J}$,
$\{ Z_{jk} \}_{j \in J, k \in N_j}$,
$\{ W_{jk} \}_{j \in J, k \in N_j}$,
$\{ V_{jk} \}_{j \in J, k \in N_j}$
and
$W$
be square integrable random variables, so that
\begin{align*}
	W &= \sum_{j \in J} X_j, &
	\E[X_j] &= 0 \;\forall j \in J, &
	\E[W^2] &= 1, \\
	W &= W_j + Z_j \;\forall j \in J, &
	Z_j &= \sum_{k \in N_j} Z_{jk}, \\
	W_j &= W_{jk} + V_{jk} \;\forall j \in J, k \in N_j.
\end{align*}
Suppose that for $j \in J, k\in N_j$
\begin{itemize}
	\item $W_j$ is independent of $X_j$,
	\item $W_{jk}$ is independent of $(X_j, Z_{jk})$.
\end{itemize}
We call $W$ to be {\bf BKR-decomposable} in this case. Let $\varphi$ be the characteristic function of $W$.
\end{defn}

Examples of decomposable random variables are presented in \cite{BKR89} and \cite{R03}. They include the notion of finite dependence used
by Chen in \cite{Chen78}, as well as the class of dissociated random variables introduced by McGinley and Sibson \cite{McGSib75}.
In the model of an Erd\"os-R\'enyi random graph, the number of copies of a given graph, the number of induced copies of this graph, the number of isolated trees of order $k \geq 2$, the number of vertices of degree $k \geq 1$ and
the number of isolated vertices are examples. Moreover Nash equilibria and linear rank statistics can be BKR-decomposed, see \cite{R03}.

\citet*{BKR89} show, that under these assumptions, there is a universal constant $C \in \R$, so that
$\Wass(\L(W), \Normal(0,1)) \leq C \epsilon$
with
\begin{align}
	\epsilon
	&=
	\frac12 \sum_{j \in J} \E[ \vert X_j \vert Z_j^2 ]
	+
	\sum_{j \in J} \sum_{k \in N_j} \big( \E[ \vert X_j Z_{jk} V_{jk} \vert ] + \E[ \vert X_j Z_{jk} \vert ] \E[ \vert Z_j + V_{jk} \vert ] \big)
	\label{BKR-epsilon}
\end{align}
(see Theorem 1 in \cite{BKR89}).

In this section, we want to derive a similar bound for the Kolmogorov distance.
To achieve this, we will use the results of the previous section.
Therefore, we need to construct a stochastic process $(H_t)_{t\in\R}$ with the properties described in \autoref{theo:STMresult}.
Especially, we need to find bounds for $\vert \E[ H_t ] \vert$ and $\vert \Cov( H_t  , e^{-itW}) \vert$.\\
For the calculations in this section, the following class of functions will be useful:
For $l \in \N$ define $R_l : \R \longrightarrow \C$ with
\begin{align}
	R_l(z) = \sum_{m=0}^\infty \frac{(iz)^m}{(m+l)!}.
	\label{defn:R}
\end{align}
The following properties are easy to prove. We will therefore omit the proofs.
\begin{rem}\label{R-properties}
	$\forall z \in \R, l \in \N$ it holds, that
	\begin{align*}
		&&\text{(i)}&&	e^{iz} &= \sum_{m=0}^{l-1} \frac{(iz)^m}{m!} + (iz)^l \cdot R_l(z),&&&&\\
		&&\text{(ii)}&&	1 &= e^{iz} - iz R_1(z),\\
		&&\text{(iii)}&&	1 &= e^{iz} - iz + z^2 R_2(z),\\
		&&\text{(iv)}&&	\Vert R_l \Vert_\infty &= R_l(0) = \frac1{l!}.
	\end{align*}
\end{rem}

Now the stochastic process $(H_t)_{t\in\R}$ can be introduced. Let $W$ be BKR-decomposable. The corresponding process $H_t$ 
is given by 
\begin{align} \label{Ht}
	H_t :={}& i\sum_{j \in J} X_j Z_j^2 R_2(-tZ_j)
			-i \sum_{j \in J} \sum_{k \in N_j}\big( X_j Z_{jk} - \E[X_j Z_{jk}] \big)
									(Z_j + V_{jk}) R_1(-t(Z_j + V_{jk})).
\end{align}

The following \autoref{BKR-Ht} shows that this definition of  $(H_t)_{t\in\R}$ indeed is consistent
with the assumptions made in the introduction.

\begin{lem}\label{BKR-Ht}
Let $W$ be BKR-decomposable.  For all $t \in \R$ it holds that
	\begin{align*}
		\E[ (iW+t) \cdot e^{itW} ] &= \E[ t^2 \cdot H_t \cdot e^{itW} ].
	\end{align*}
	Here, $H_t$ is defined in \eqref{Ht}.
\end{lem}
\begin{proof}
	The strategy of this proof relies on a combination of strategies by
	\cite{BKR89}
	and
	\cite{Ro17}.
	It consists of several decompositions of $1$ together with Taylor expansions.
	First, we find
	\begin{align}
		iW e^{itW}
		&=
		iW e^{itW} \cdot 1 \nonumber\\
		&=
		i\sum_{j \in J} X_j e^{itW} \big( e^{-itZ_j} - i (-tZ_j) + t^2 Z_j^2 R_2(-tZ_j) \big) \nonumber\\
		&=
		i\sum_{j \in J} X_j e^{itW_j} - t \sum_{j \in J}X_j Z_j e^{itW} + it^2 \sum_{j \in J} X_j Z_j^2 R_2(-tZ_j) e^{itW} \label{eq:1}
		.
	\end{align}
	Similar calculation steps are applied to the second summand of \eqref{eq:1}:
	\begin{align}
		&-t \sum_{j \in J} X_j Z_j e^{itW} \nonumber\\
		={}&
		-t \sum_{j \in J} \sum_{k \in N_j} X_j Z_{jk} e^{itW} \cdot 1 \nonumber\\
		={}&
		-t \sum_{j \in J} \sum_{k \in N_j} X_j Z_{jk} e^{itW} \big( e^{-it(Z_j + V_{jk})} + it(Z_j + V_{jk}) R_1(-t(Z_j + V_{jk})) \big) \nonumber\\
		={}&
		-t \sum_{j \in J} \sum_{k \in N_j} X_j Z_{jk} e^{itW_{jk}}
			- it^2 \sum_{j \in J} \sum_{k \in N_j} X_j Z_{jk} (Z_j + V_{jk}) R_1(-t(Z_j + V_{jk})) e^{itW} \label{eq:2}
		.
	\end{align}
	It remains to analyze $t e^{itW}$:
	Using the decomposition
	\begin{align*}
		1 &= \E[W^2]
		= \sum_{j \in J} \E[X_j (W_j + Z_j) ]
		= \sum_{j \in J} \E[X_j] \E[W_j] + \sum_{j \in J} \sum_{k \in N_j} \E[X_j Z_{jk} ]\\
		&= \sum_{j \in J} \sum_{k \in N_j} \E[X_j Z_{jk} ]
	\end{align*}
	we find
	\begin{align}
		t e^{itW}
		={}& t \sum_{j \in J} \sum_{k \in N_j} \E[X_j Z_{jk} ] e^{itW} \cdot 1 \nonumber\\
		={}& t \sum_{j \in J} \sum_{k \in N_j} \E[X_j Z_{jk} ] e^{itW}
									\big( e^{-it(Z_j + V_{jk})} + it(Z_j + V_{jk}) R_1(-t(Z_j + V_{jk})) \big) \nonumber\\
		={}& t \sum_{j \in J} \sum_{k \in N_j} \E[X_j Z_{jk} ] e^{itW_{jk}} \nonumber\\
			&+ it^2 \sum_{j \in J} \sum_{k \in N_j} \E[X_j Z_{jk} ](Z_j + V_{jk}) R_1(-t(Z_j + V_{jk})) e^{itW} \label{eq:3}
		.
	\end{align}
	All previous results \eqref{eq:1}, \eqref{eq:2}, \eqref{eq:3} together lead to
	\begin{align*}
		& (iW+t) \cdot e^{itW}\\
		={}&
		i\sum_{j \in J} X_j e^{itW_j}
		- t \sum_{j \in J} \sum_{k \in N_j} X_j Z_{jk} e^{itW_{jk}}\\
		&+ it^2 \sum_{j \in J} \sum_{k \in N_j} X_j Z_{jk} (Z_j + V_{jk}) R_1(-t(Z_j + V_{jk})) e^{itW}
		- it^2 \sum_{j \in J} X_j Z_j^2 R_2(-tZ_j) e^{itW}\\
		&+ t \sum_{j \in J} \sum_{k \in N_j} \E[X_j Z_{jk} ] e^{itW_{jk}}
		+ it^2 \sum_{j \in J} \sum_{k \in N_j} \E[X_j Z_{jk} ](Z_j + V_{jk}) R_1(-t(Z_j + V_{jk})) e^{itW}\\
		={}&
		i\sum_{j \in J} X_j e^{itW_j}
		- t \sum_{j \in J} \sum_{k \in N_j} \big( X_j Z_{jk} - \E[X_j Z_{jk}] \big) e^{itW_{jk}}
		+ t^2 \cdot H_t \cdot e^{itW}
		.
	\end{align*}
	Since $X_j$ is centered and $X_j$ and $W_j$ respectively $X_jZ_{jk}$ and $W_{jk}$ are independent,
	we get the result $\E[ (iW+t) \cdot e^{itW} ] = \E[ t^2 \cdot H_t \cdot e^{itW} ]$.
\end{proof}

We are now able to show the current section's main result:
\begin{theo}\label{BKR-Schranken}
Let $W$ be BKR-decomposable, then for all $t\in\R$ it holds that
	\begin{align*}
	\E[ \vert H_t \vert ]
		\leq{}& \frac12 \sum_{j \in J} \E[ \vert X_j \vert Z_j^2 ] \\
			&+ \sum_{j \in J} \sum_{k \in N_j}
				\big( \E[ \vert X_j Z_{jk} (Z_j + V_{jk}) \vert ] + \E[ \vert X_j Z_{jk} \vert ] \E[ \vert Z_j + V_{jk} \vert ] \big),\\
	\vert \Cov( H_t  , e^{-itW}) \vert
		\leq{}& \Big( \Var\Big(\sum_{j \in J} X_j Z_j^2 R_2(-tZ_j) \Big) \Big)^{1/2}\\
			& + \Big( \Var\Big(\sum_{j \in J} \sum_{k \in N_j} X_j Z_{jk} (Z_j + V_{jk}) R_1(-t(Z_j + V_{jk})) \Big) \Big)^{1/2}\\
			& + \Big( \Var\Big(\sum_{j \in J} \sum_{k \in N_j} \E[ X_j Z_{jk} ] (Z_j + V_{jk}) R_1(-t(Z_j + V_{jk})) \Big) \Big)^{1/2},
	\end{align*}
	where $R_1$ and $R_2$ are defined in \eqref{defn:R} and $H_t$ is defined in \eqref{Ht}.
\end{theo}
Note, that the upper bound for $\E[ \vert H_t \vert ]$ in \autoref{BKR-Schranken}
is very similar but not identical to \eqref{BKR-epsilon}, which is the bound used in \cite{BKR89}.
\begin{proof}[Proof of \autoref{BKR-Schranken}]
	The bound for $\E[ \vert H_t \vert ]$ is a direct consequence of the definition of $H_t$
	and the properties of $R_1$ and $R_2$ from \autoref{R-properties}.
	It remains to prove the bound for $\Cov(H_t , e^{-itW})$.
	Since the covariance is a sesquilinear form, we get
	\begin{align*}
		&\Cov( H_t , e^{-itW})\\
		={}& i \Cov\Big( \sum_{j \in J} X_j Z_j^2 R_2(-tZ_j), e^{-itW} \Big)\\
			& - i \Cov\Big( \sum_{j \in J} \sum_{k \in N_j}
						X_j Z_{jk} (Z_j + V_{jk}) R_1(-t(Z_j + V_{jk}))
								 , e^{-itW} \Big)\\
			& + i \Cov\Big( \sum_{j \in J} \sum_{k \in N_j}
						\E[ X_j Z_{jk} ] (Z_j + V_{jk}) R_1(-t(Z_j + V_{jk}))
								 , e^{-itW} \Big)
		.
	\end{align*}
	Now, note that $\Var(e^{itW}) \leq 1$.
	Application of the Cauchy-Schwarz inequality therefore yields the desired result.
\end{proof}
The purpose of \autoref{BKR-Schranken} is to plug these upper bounds into \autoref{theo:STMresult}.
Therefore, we need to make sure that the assumptions of \autoref{theo:STMresult} are fulfilled.
It remains to check
whether $H_t$ is integrable $\forall t \in \R$
and whether the mappings
$t \mapsto t^2 \cdot \E[H_t]$ and $t \mapsto t^2 \cdot \Cov(H_t , e^{-itW})$, $t \in \R$,
are continuous.

The integrability of $H_t$ for $t \neq 0$ already follows from the proof of the preceeding \autoref{BKR-Ht}.
The integrability of $H_0$ though does not follow from this lemma.
However, with similar arguments one can show,
that the integrability of $H_0$ is equivalent to the integrability of $W^3$, which is proved below in \autoref{W3}.
As already stated, the continuity of
$t \mapsto t^2 \cdot \E[H_t]$ and $t \mapsto t^2 \cdot \Cov(H_t , e^{-itW})$, $t \in \R$,
is not necessary.
So the existence of the third moment of $W$ would be a sufficient condition
to use \autoref{theo:STMresult}.\\
However, we want to note that there is a simpler way to make sure that \autoref{theo:STMresult} may be used:
In this section's main result, \autoref{BKR-Schranken}, the expectation of the quantity
\begin{align}
	\frac12 \sum_{j \in J} \vert X_j \vert Z_j^2
		+ \sum_{j \in J} \sum_{k \in N_j}
			\big( \vert X_j Z_{jk} (Z_j + V_{jk}) \vert + \E[ \vert X_j Z_{jk} \vert ] \vert Z_j + V_{jk} \vert \big)
	\label{eq:upperbound}
\end{align}
appears as an upper bound.
In cases, in which this expectation does not exist, all results in this paper are trivial.
So from now on we may assume without loss of generality, that the expectation of \eqref{eq:upperbound} exists.
However, \eqref{eq:upperbound} is an uniform bound for $H_t$, $t \in \R$.
Therefore, all $H_t$ are integrable. Further, by dominated convergence the mappings
$t \mapsto t^2 \cdot \E[H_t]$ and $t \mapsto t^2 \cdot \Cov(H_t , e^{-itW})$, $t \in \R$,
are continuous.
Thus, all assumptions from \autoref{theo:STMresult} are fulfilled.

\begin{prop}\label{W3}
Let $W$ be BKR-decomposable. Then $H_0$, as defined in \eqref{Ht}, is integrable, if and only if $W^3$ is integrable.
\end{prop}
\begin{proof}
	We use the decomposition
	\begin{align*}
		W^3
		&= \sum_{j \in J} X_j W^2 \\
		&= \sum_{j \in J} X_j W_j^2 + 2\sum_{j \in J} X_j Z_j W - \sum_{j \in J} X_j Z_j^2 \\
		&= \sum_{j \in J} X_j W_j^2 + 2\sum_{j \in J} \sum_{k \in N_j} X_j Z_{jk} W - \sum_{j \in J} X_j Z_j^2 \\
		&= \sum_{j \in J} X_j W_j^2 + 2\sum_{j \in J} \sum_{k \in N_j} X_j Z_{jk} W_{jk}
				 + 2\sum_{j \in J} \sum_{k \in N_j} X_j Z_{jk} (Z_j + V_{jk}) - \sum_{j \in J} X_j Z_j^2.
	\end{align*}
	Additionally, we find
	\begin{align*}
		W
		&= \sum_{j \in J} \sum_{k \in N_j} \E[ X_j Z_{jk} ] W \\
		&= \sum_{j \in J} \sum_{k \in N_j} \E[ X_j Z_{jk} ] (Z_j + V_{jk})
			+ \sum_{j \in J} \sum_{k \in N_j} \E[ X_j Z_{jk} ] W_{jk}.
	\end{align*}
	This leads to
	\begin{align*}
		\frac12 W^3 - W
		={}& \frac12 \sum_{j \in J} X_j W_j^2
			+ \sum_{j \in J} \sum_{k \in N_j} X_j Z_{jk} W_{jk}
			+ \sum_{j \in J} \sum_{k \in N_j} X_j Z_{jk} (Z_j + V_{jk})
			- \frac12 \sum_{j \in J} X_j Z_j^2\\
			&- \sum_{j \in J} \sum_{k \in N_j} \E[ X_j Z_{jk} ] (Z_j + V_{jk})
			- \sum_{j \in J} \sum_{k \in N_j} \E[ X_j Z_{jk} ] W_{jk}\\
		={}& i \cdot H_0
			+ \frac12 \sum_{j \in J} X_j W_j^2
			+ \sum_{j \in J} \sum_{k \in N_j} \big( X_j Z_{jk} - \E[ X_j Z_{jk} ] \big) W_{jk}.
	\end{align*}
	Since $X_j$ is centered and $X_j$ and $W_j$ respectively $X_jZ_{jk}$ and $W_{jk}$ are independent,
	we get as result, that there is $\E[\vert H_0 \vert] < \infty$, if and only if $\E[\vert W^3\vert] < \infty$.
\end{proof}

Let us summarize our results. We formulate a plug-in Theorem for any $W$ which is assumed to be BKR-decomposable.

\begin{theo}
Let $W$ be a BKR-decomposable random variable. Then 
\begin{align*}
		\Kol(\L(W), \Normal(0,1))
		&\leq
		\Big( \frac4{3\sqrt\pi} + \frac{24 \sqrt2}{\pi\sqrt{\pi}} \Big) \cdot A + \frac2\pi \cdot \frac{B}{A},
	\end{align*}
	where 
	$$
	A := \frac12 \sum_{j \in J} \E[ \vert X_j \vert Z_j^2 ] \\
				+ \sum_{j \in J} \sum_{k \in N_j}
					\big( \E[ \vert X_j Z_{jk} (Z_j + V_{jk}) \vert ] + \E[ \vert X_j Z_{jk} \vert ] \E[ \vert Z_j + V_{jk} \vert ] \big)
	$$
	and 
	\begin{align*}
	B := & \Big( \Var\Big(\sum_{j \in J} X_j Z_j^2 R_2(-tZ_j) \Big) \Big)^{1/2}\\
			& + \Big( \Var\Big(\sum_{j \in J} \sum_{k \in N_j} X_j Z_{jk} (Z_j + V_{jk}) R_1(-t(Z_j + V_{jk})) \Big) \Big)^{1/2}\\
			& + \Big( \Var\Big(\sum_{j \in J} \sum_{k \in N_j} \E[ X_j Z_{jk} ] (Z_j + V_{jk}) R_1(-t(Z_j + V_{jk})) \Big) \Big)^{1/2}.
	\end{align*}
	$R_1$ and $R_2$ are defined in \eqref{defn:R}.
\end{theo}

\section{Proof of Theorem 1.1}\label{sec:MT}

We can now proof our main application, \autoref{maintheorem}.
Let $G(n,p)$ be an Erdös-Rényi random graph with $n \in \N$ vertices
and probability $p=p(n)$, which may depend on $n$.
Let $E$ be the set of all possible edges of $G(n,p)$.
Finally, let $\{ I_l \}_{l \in E}$ be a family of independent indicator random variables, where $I_l=1$ indicates the presence of edge $l \in E$.
Obviously, $I_l \sim \text{Bernoulli}(p)$.

From now on, we suppose that all graphs and subgraphs in the following argumentation do not have isolated vertices.
Therefore, every (sub-)graph can be uniquely identified by its set of edges.
So there is a natural bijection between all possible subgraphs of $G(n,p)$ and the power set (without the empty set)
$\mathcal P (E) \backslash\{\emptyset\}$.
Thus, we will not distinguish in notation between a subset of $E$ and a (sub-)graph.
This implies further that every graph in our argumentation has at least one edge.
We will later be able to drop this assumption.

Given a graph $G$, let
$v_G$ be the number of vertices of $G$
and
$e_G$ be the number of edges of $G$.
Now, let $\G$ be a fixed graph.
Given a subgraph $H \subset \G$, define
\begin{align*}
	\Psi_H = n^{\nu_H} \cdot p^{e_H} \,\, \text{and} \,\,
	\Psi = \min\limits_{H \subset \G} \Psi_H.
\end{align*}
Since we assume, that every subgraph has at least one edge and therefore at least two vertices, we find that $\Psi \leq n^2 \cdot p$.\\
Further, we define the set $J := \{ j \subset E \;\vert\;$the graph given by $j$ is isomorphic to $\G \}.$
For every $j \in J$, we define the random variables
\begin{align*}
	Y_j = \prod_{m \in j} I_m,
\end{align*}
which is Bernoulli($p^{e_j}$)-distibuted and indicates, whether the graph $j$ is a subgraph of the random graph $G(n,p)$.
Further, let
\begin{align*}
	X_j = \sigma^{-1} \cdot (Y_j - \E[Y_j]) \,\, \text{with} \,\,
\sigma^2 := \Var \big( \sum_{j \in J} Y_j \big).
\end{align*}

Now we see that the family $\{ X_j \}_{j \in J}$ fulfills the BKR-decomposition properties:
With
\begin{align*}
	W &= \sum_{j \in J} X_j
\end{align*}
it follows, that $\E[X_j] = 0$ for all $j \in J$ and $\E[W^2] = 1$. For each $j \in J$ the neighborhood of $X_j$ is given by
\begin{align*}
	N_j := \{ k \in J \;\vert\; X_j \text{ and } X_k \text{ are dependent} \}.
\end{align*}
Further, let $D := \max\limits_{j \in J} \vert N_j \vert$
denote the cardinality of the biggest neighborhood.
Note that two random variables $X_j$ and $X_k$ are independent
if and only if the sets of edges $j$ and $k$ are disjoint respectively the graphs $j$ and $k$ do not have any edge in common.\\
The neighborhoods of higher order
\begin{flalign*}
	&&	N_j^{(1)}	&= N_j							&	\forall j &\in J, \\
	&&	N_j^{(m+1)} &= \bigcup_{k \in N_j^{(m)}} N_k		&	\forall j \in J, m &\in \N,
\end{flalign*}
as well as this set of chains of connected indices
\begin{align*}
	J_c= \{
		(j_1,j_2,j_3,j_4,j_5,j_6) \in J^6
		\;\vert\;
		&j_2 \in N_{j_1},\\
		&j_3 \in N_{j_1} \cup N_{j_2},\\
		&\dots, \\
		&j_6 \in N_{j_1} \cup N_{j_2} \cup N_{j_3} \cup N_{j_4} \cup N_{j_5}
	\}
\end{align*}
will be useful, later. Since all our random variables are bounded, they are also integrable and all moments exist.
In this case, we derive the BKR-decomposition from the previous \autoref{sec:BKR} in the following way:
For every $j \in J$, $k \in N_j$ we choose
\begin{align*}
	Z_{jk} = X_k, \,\,
	V_{jk} = \sum_{l \in N_j^C \cap N_k} X_l,\,\,
	W_{jk} = \sum_{l \in N_j^C \cap N_k^C} X_l,
\end{align*}
so that
\begin{align*}
	Z_j = \sum_{k \in N_j} X_k \,\, \text{and} \,\,
	W_j = \sum_{k \in N_j^C} X_k.
\end{align*}
This eventually leads to
\begin{align*}
	H_t
	={}& i\sum_{j \in J} X_j Z_j^2 R_2(-tZ_j)
			-i \sum_{j \in J} \sum_{k \in N_j}\big( X_j Z_{jk} - \E[X_j Z_{jk}] \big)
									(Z_j + V_{jk}) R_1(-t(Z_j + V_{jk}))\\
	={}& i
		\sum_{\mathclap{\substack{j \in J \\ k, l \in N_j}}}
			X_j X_k X_l
			R_2 \Big(-t \sum_{\mathclap{m \in N_j}} X_m \Big)
		-i
		\sum_{\mathclap{\substack{j \in J \\ k \in N_j, l \in N_j \cup N_k}}}
			\big( X_j X_k - \E[X_j X_k] \big)
			X_l
			R_1 \Big(-t \sum_{\mathclap{m \in N_j \cup N_k}} X_m \Big)
\end{align*}
Again, since all $X_j$ are bounded, all $H_t$ are integrable.

Now all assumptions from the previous sections are fulfilled so that we may apply
\autoref{BKR-Schranken}
to construct some upper bounds for $\E[H_t]$ and $\Cov(H_t , e^{-itW})$.
These upper bounds will be plugged into
\autoref{theo:STMresult}
to prove our main theorem (\autoref{maintheorem}).
Note that the upper bound in our main theorem distinguishes two cases.
Therefore, we will have to construct two bounds for $\E[H_t]$ and $\Cov(H_t , e^{-itW})$ each.
Using \autoref{BKR-Schranken}, we want to prove the following upper bounds.
\begin{lem}\label{sc-SchrankenH}
	For every $t \in \R$ it holds, that
	\begin{align*}
		\vert \E[H_t] \vert &\leq \frac{9 D^2}{2 \sigma^2} \cdot \sum_{j \in J} \E[\vert X_j \vert],\\
		\vert \E[H_t] \vert &\leq \frac{36}{\sigma^3} \cdot \sum_{\substack{j \in J \\ k, l \in N_j}} \E[ Y_j Y_k Y_l ].
	\end{align*}
\end{lem}
\begin{lem}\label{sc-SchrankenC}
	For every $t \in \R$ it holds, that
	\begin{align*}
		\vert \Cov(H_t , e^{-itW}) \vert &\leq \frac{20 D^\frac52}{\sigma^\frac52}
			\cdot \Big( \sum_{j \in J} \E[\vert X_j \vert] \Big)^{1/2},\\
		\vert \Cov(H_t , e^{-itW}) \vert &\leq \frac{113}{\sigma^3} \cdot
			\Bigl( \sum_{(j_1,j_2,j_3,j_4,j_5,j_6) \in J_c}
			\E[ Y_{j_1} Y_{j_2} Y_{j_3} Y_{j_4} Y_{j_5} Y_{j_6} ] \Big)^{1/2}.
	\end{align*}
\end{lem}

Before we show the proof of the inequalities given in \autoref{sc-SchrankenH} and \autoref{sc-SchrankenC},
we want to give a short explanation of some calculation steps, which will repeatedly appear.
We will exemplify these strategies in the proof of the following \autoref{sc-EWs-connect}.
In this list all, the inequalities are exemplified for a fixed number of random variables involved.
\begin{lem}\label{sc-EWs-connect}\mbox{}
		(i) It holds, that
			\begin{align*}
				\sum_{\substack{j \in J \\ k, l \in N_j}} \E[ \vert X_j X_k X_l \vert]
				&\leq
				D^2 \sum_{j \in J} \E[ \vert X_j \vert^3 ],\\
				\sum_{\substack{j \in J \\ k, l \in N_j}} \E[ \vert X_j X_k \vert] \E[ \vert X_l \vert]
				&\leq
				D^2 \sum_{j \in J} \E[ \vert X_j \vert^3 ].
			\end{align*}
		(ii) For $j,k,l,m \in J$, it holds, that
			\begin{align}
				\E[ Y_j Y_k ] \E[ Y_l Y_m ] &\leq \E[ Y_j Y_k Y_l Y_m ], \label{list1}\\
				\E[ \vert X_j X_k \vert ] &\leq \big(\frac2\sigma\big)^2 \E[ Y_j Y_k ], \label{list2}\\
				\vert \Cov (Y_j  Y_k, Y_l Y_m) \vert &\leq 2 \E[ Y_j Y_k Y_l Y_m ], \label{list3}\\
				\vert \Cov (X_j X_k, X_l X_m) \vert &\leq 2 \big(\frac2\sigma\big)^4 \E[Y_j Y_k Y_l Y_m], \label{list4}\\
				\vert \Cov (X_j X_k, X_l X_m) \vert &\leq \frac2{\sigma^3} \E[\vert X_j \vert]. \label{list5}
			\end{align}
\end{lem}
\begin{proof}[Proof of \autoref{sc-EWs-connect}]\mbox{}
		(i) $\forall j, k, l \in J$ the application
			of the inequality of arithmetic and geometric means
			and for the second line also the application of Jensen's inequality
			lead to
			\begin{align*}
				\E[ \vert X_j X_k X_l \vert]
					&\leq \frac13 \big( \E[ \vert X_j \vert^3 ] + \E[ \vert X_k \vert^3 ] + \E[ \vert X_l \vert^3 ] \big),\\
				\E[ \vert X_j X_k \vert] \E[ \vert X_l \vert]
					&\leq \frac13 \big( \E[ \vert X_j \vert^3 ] + \E[ \vert X_k \vert^3 ] + \E[ \vert X_l \vert^3 ] \big).
			\end{align*}
			Now, we use symmetries of the dependency graph structure.
			For example, we use, that $k \in N_j \Leftrightarrow j \in N_k$.
			Thus,
			\begin{align*}
				\sum_{\substack{j \in J \\ k,l  \in N_j}} \frac13 \big( \E[ \vert X_j \vert^3 ] + \E[ \vert X_k \vert^3 ] + \E[ \vert X_l \vert^3 ] \big) 
				= \sum_{\substack{j \in J \\ k,l \in N_j}} \E[ \vert X_j \vert^3 ] 
				\leq D^2 \sum_{j \in J} \E[ \vert X_j \vert^3 ].
			\end{align*}
		(ii)	It holds, that
			\begin{align*}
				\E[ Y_j Y_k ] \E[ Y_l Y_m ]
				&= \E[ \prod_{a \in j \cup k } I_a ] \E[ \prod_{a \in l \cup m} I_a ]
				= p^{\vert j \cup k \vert} \cdot p^{\vert l \cup m \vert}
				\leq p^{\vert j \cup k \cup l \cup m \vert}
				= \E[ Y_j Y_k Y_l Y_m ].
			\end{align*}
			This proves \eqref{list1}. Since $X_j = \frac1\sigma \cdot (Y_j - \E[Y_j])$,
			by expansion of the product we find
			\begin{align*}
				\E[ X_j X_k ]
				\leq \frac1{\sigma^2} \cdot \E[ (Y_j + \E[Y_j]) \cdot (Y_k + \E[Y_k]) ] 
				= \frac1{\sigma^2} \cdot \big( \E[ Y_j Y_k ] + 3 \E[Y_j]\E[Y_k] \big).
			\end{align*}
			\eqref{list2} then is implied by an argument similar to the one used in \eqref{list1}. It holds, that
			$\vert \Cov (Y_j Y_k , Y_l Y_m) \vert \leq \E[ Y_j Y_k Y_l Y_m ] + \E[Y_j Y_k ]\E[ Y_l Y_m]$.
			\eqref{list3} then follows from \eqref{list1}. Moreover, we have
			$\vert \Cov (X_j X_k , X_l X_m) \vert \leq \E[ X_j X_k X_l X_m ] + \E[X_j X_k ]\E[ X_l X_m]$.
			\eqref{list4} is implied by arguments similar to those used in \eqref{list1} and \eqref{list2}. Finally, we have
			$\vert \Cov (X_j X_k , X_l X_m) \vert \leq \E[ X_j X_k X_l X_m ] + \E[X_j X_k ]\E[ X_l X_m]$.
			Since $X_k = \frac1\sigma \cdot (Y_k - \E[Y_k])$,
			we know that $\vert X_k \vert \leq \frac1\sigma$.
			Thus,
			\begin{align*}
				\E[ X_j X_k X_l X_m ] + \E[X_j X_k ]\E[ X_l X_m]
				&\leq \frac1{\sigma^3} \E[\vert X_j \vert] + \frac1{\sigma^3} \E[\vert X_j \vert].
			\end{align*}
			This proves \eqref{list5}.
\end{proof}
\begin{proof}[Proof of \autoref{sc-SchrankenH}]
	\autoref{BKR-Schranken} implies, that
	\begin{align}
		\vert \E[H_t] \vert
		\leq{}& \frac12 \sum_{j \in J} \E[ \vert X_j \vert Z_j^2 ]
			+
			\sum_{j \in J} \sum_{k \in N_j}
				\big(
					\E[ \vert X_j Z_{jk} (Z_j + V_{jk}) \vert ]
					+
					\E[ \vert X_j Z_{jk} \vert ] \E[ \vert Z_j + V_{jk} \vert ]
				\big) \nonumber\\
		\leq{}& \frac12 \sum_{\substack{j \in J \\ k,l \in N_j}}
				\E[ \vert X_j X_k X_l \vert ]
			+
			\sum_{\substack{j \in J \\ k \in N_j, l \in N_j \cup N_k}}
				\big(
					\E[ \vert X_j X_k X_l \vert ]
					+
					\E[ \vert X_j X_k \vert ]\E[ \vert X_l \vert ]
				\big). \label{proofHsplit}
	\end{align}
	On the one hand, according to the first part of \autoref{sc-EWs-connect}, \eqref{proofHsplit} can be bounded by
	\begin{align*}
       \frac12 \cdot D^2 \cdot \sum_{j \in J} \E[ \vert X_j \vert^3 ]
			+ D \cdot 2D \cdot \sum_{j \in J} \E[ \vert X_j \vert^3 ]
			+ D \cdot 2D \cdot \sum_{j \in J} \E[ \vert X_j \vert^3 ]
		= \frac92 \cdot D^2 \cdot \sum_{j \in J} \E[ \vert X_j \vert^3 ],
	\end{align*}
	which leads to the upper bound $\frac{9}{2} \cdot \frac{D^2}{\sigma^2} \cdot \sum_{j \in J} \E[ \vert X_j \vert ]$.
	This proves the first inequality from \autoref{sc-SchrankenH}.
	
	On the other hand, we can use some properties from the second part of \autoref{sc-EWs-connect}
	to derive another upper bound for \eqref{proofHsplit}:
	With $\E[ \vert X_j X_k X_l \vert ] \leq \big(\frac2\sigma\big)^3 \E[ Y_j Y_k Y_l ]$
	and \linebreak $\E[ \vert X_j X_k \vert ]\E[ \vert X_l \vert ] \leq \big(\frac2\sigma\big)^3 \E[ Y_j Y_k Y_l ]$,
	\eqref{proofHsplit} can be bounded by
	\begin{align*}
		& \big(\frac2\sigma\big)^3 \Big(
			\frac12 \sum_{\substack{j \in J \\ k,l \in N_j}}
				\E[ Y_j Y_k Y_l ]
			+
			2 \sum_{\substack{j \in J \\ k \in N_j, l \in N_j \cup N_k}}
				\E[ Y_j Y_k Y_l ]
			\Big)\\
		\leq{}& \big(\frac2\sigma\big)^3 \Big(
			\frac12 \sum_{\substack{j \in J \\ k,l \in N_j}}
				\E[ Y_j Y_k Y_l ]
			+
			2 \sum_{\substack{j \in J \\ k,l \in N_j}}
				\E[ Y_j Y_k Y_l ]
			+
			2 \sum_{\substack{j \in J \\ k \in N_j, l \in N_k}}
				\E[ Y_j Y_k Y_l ]
			\Big)\\
		={}& \frac{36}{\sigma^3}
			\sum_{\substack{j \in J \\ k,l \in N_j}}
			\E[ Y_j Y_k Y_l ]
			.
	\end{align*}
	Hence, both inequalities from \autoref{sc-SchrankenH} are proven.
\end{proof}
\begin{proof}[Proof of \autoref{sc-SchrankenC}]
	\autoref{BKR-Schranken} states, that
	\begin{align*}
		\vert \Cov( H_t  , e^{-itW}) \vert
		\leq{}& \Big( \Var\Big(\sum_{j \in J} X_j Z_j^2 R_2(-tZ_j) \Big) \Big)^{1/2}\\
			& + \Big( \Var\Big(\sum_{j \in J} \sum_{k \in N_j} X_j Z_{jk} (Z_j + V_{jk}) R_1(-t(Z_j + V_{jk})) \Big) \Big)^{1/2}\\
			& + \Big( \Var\Big(\sum_{j \in J} \sum_{k \in N_j} \E[ X_j Z_{jk} ] (Z_j + V_{jk}) R_1(-t(Z_j + V_{jk})) \Big) \Big)^{1/2}.
	\end{align*}
	For every $t \in \R$ we introduce new functions
	\begin{align*}
		f_{1,t}: \R &\longrightarrow \R, &
		f_{2,t}: \R &\longrightarrow \R, \\
		x &\longmapsto x R_1(-tx), &
		x &\longmapsto x R_2(-tx).
	\end{align*}
	Hence, we can write the upper bound from \autoref{BKR-Schranken} as
	\begin{align}
		\vert \Cov( H_t  , e^{-itW}) \vert
		\leq{}& \Big( \Var\Big(\sum_{j \in J} X_j Z_j f_{2,t}(Z_j) \Big) \Big)^{1/2} \nonumber\\
			& + \Big( \Var\Big(\sum_{j \in J} \sum_{k \in N_j} X_j Z_{jk} f_{1,t}(Z_j + V_{jk}) \Big) \Big)^{1/2} \nonumber\\
			& + \Big( \Var\Big(\sum_{j \in J} \sum_{k \in N_j} \E[ X_j Z_{jk} ] f_{1,t}(Z_j + V_{jk}) \Big) \Big)^{1/2} \nonumber\\
		={}& \Big( \Var\Big(
				\sum_{\mathclap{\substack{j \in J \\ k \in N_j}}}
				X_j X_k
				f_{2,t} \Big( \sum_{m \in N_j} X_m \Big)
			\Big) \Big)^{1/2} \label{scA1}\\
			& + \Big( \Var\Big(
				\sum_{\mathclap{\substack{j \in J \\ k \in N_j}}}
				X_j X_k
				f_{1,t} \Big( \sum_{m \in N_j \cup N_k} X_m \Big)
			\Big) \Big)^{1/2} \label{scA2}\\
			& + \Big( \Var\Big(
				\sum_{\mathclap{\substack{j \in J \\ k \in N_j}}}
				\E[X_j X_k]
				f_{1,t} \Big( \sum_{m \in N_j \cup N_k} X_m \Big)
			\Big) \Big)^{1/2}. \label{scA3}
	\end{align}
	We will analyze each summand \eqref{scA1}, \eqref{scA2}, \eqref{scA3} separately.
	We start with \eqref{scA3}, since it is the easiest one to analyze,
	but the strategy can later be extended to be useful for the other two summands, too.

	\noindent\textbf{Term \eqref{scA3}:} $\displaystyle\Big( \Var\Big(
				\sum_{\mathclap{\substack{j \in J \\ k \in N_j}}}
				\E[X_j X_k]
				f_{1,t} \Big( \sum_{m \in N_j \cup N_k} X_m \Big)
			\Big) \Big)^{1/2}$\\
	We have
	\begin{align*}
		&\Var\Big(
				\sum_{\mathclap{\substack{j \in J \\ k \in N_j}}}
				\E[X_j X_k]
				f_{1,t} \Big( \sum_{m \in N_j \cup N_k} X_m \Big)
			\Big) \\
		={}& \sum_{\mathclap{\substack{ j_1 \in J, k_1 \in N_{j_1} \\ j_2 \in J, k_2 \in N_{j_2} }}} 
				\E[ X_{j_1} X_{k_1} ] \E[ X_{j_2} X_{k_2} ]
				\Cov\Big( f_{1,t}\Big( \sum_{m_1 \in N_{j_1} \cup N_{k_1}} X_{m_1} \Big)
					, f_{1,t}\Big( \sum_{m_2 \in N_{j_2} \cup N_{k_2}} X_{m_2} \Big) \Big).
	\end{align*}
	Since the indicator variables $I_l$, $l \in E$, are independent, they are associated.
	Therefore, according to the properties listed in \cite{EPW67},
	the random variables
	$\sum_{m_1 \in N_{j_1} \cup N_{k_1}} X_{m_1}$
	and
	$\sum_{m_2 \in N_{j_2} \cup N_{k_2}} X_{m_2}$
	are associated, too.
	And according to Theorem~4.4 from \cite{EPW67}, they are \emph{positive quadrant dependent}.
	Thus, we may use Lemma~3 from \cite{Ne80} and we get
	\begin{align*}
		&\Big\vert \Cov\Big( f_{1,t}\Big( \sum_{m_1 \in N_{j_1} \cup N_{k_1}} X_{m_1} \Big)
					, f_{1,t}\Big( \sum_{m_2 \in N_{j_2} \cup N_{k_2}} X_{m_2} \Big) \Big) \Big\vert \\
		&\leq \Vert f^\prime_{1,t} \Vert_\infty^2 \cdot \Cov\Big( \sum_{m_1 \in N_{j_1} \cup N_{k_1}} X_{m_1}
					, \sum_{m_2 \in N_{j_2} \cup N_{k_2}} X_{m_2} \Big)\\
		&= \Vert f^\prime_{1,t} \Vert_\infty^2 \cdot
					\sum_{\substack{m_1 \in N_{j_1} \cup N_{k_1} \\ m_2 \in N_{j_2} \cup N_{k_2}}} \Cov\Big( X_{m_1} , X_{m_2} \Big).
	\end{align*}
	If $m_2 \not\in N_{m_1}$, the covariance of $X_{m_1}$ and $X_{m_2}$ is $0$.
	These summands therefore may be omitted.
	Besides, it is easy to calculate, that $\Vert f^\prime_{1,t} \Vert_\infty = 1$.
	This leads to the inequality
	\begin{align*}
		\Big\vert \Var\Big(
				\sum_{\mathclap{\substack{j \in J \\ k \in N_j}}}
				\E[X_j X_k]
				f_{1,t} \Big( \sum_{m \in N_j \cup N_k} X_m \Big)
			\Big) \Big\vert
		\leq{}& \sum_{\mathclap{\substack{j_1 \in J, k_1 \in N_{j_1} \\ m_1 \in N_{j_1} \cup N_{k_1} \\
							j_2 \in J, k_2 \in N_{j_2}, m_2 \in (N_{j_2} \cup N_{k_2}) \cap N_{m_1}}}}
			\E[ \vert X_{j_1} X_{k_1} \vert] \E[ \vert X_{j_2} X_{k_2} \vert ] \vert\Cov( X_{m_1} , X_{m_2} )\vert.
	\end{align*}
	Using some calculations, which have been exemplified in \autoref{sc-EWs-connect}, we find on the one hand, that the right hand side of the
	last estimate can be bounded by
	\begin{align*}
\quad \quad \frac{2}{\sigma^5} \quad \sum_{\mathclap{\substack{j_1 \in J, k_1 \in N_{j_1} \\ m_1 \in N_{j_1} \cup N_{k_1} \\
							j_2 \in J, k_2 \in N_{j_2}, m_2 \in (N_{j_2} \cup N_{k_2}) \cap N_{m_1}}}}
					\E[ \vert X_{j_1} \vert ]
		\leq \frac{2}{\sigma^5} \cdot D \cdot 2D \cdot D \cdot 2D \cdot D \cdot \sum_{j \in J}
					\E[ \vert X_j \vert ]
		\leq 8 \cdot \frac{D^5}{\sigma^5} \cdot \E[ \vert X_j \vert ]
	\intertext{and on the other hand by}
\quad \quad \frac{2^8}{\sigma^6} \quad \sum_{\mathclap{\substack{j_1 \in J, k_1 \in N_{j_1} \\ m_1 \in N_{j_1} \cup N_{k_1} \\
							j_2 \in J, k_2 \in N_{j_2}, m_2 \in (N_{j_2} \cup N_{k_2}) \cap N_{m_1}}}}
			\E[ Y_{j_1} Y_{k_1} Y_{m_1} Y_{j_2} Y_{k_2} Y_{m_2} ]
		\leq \frac{256}{\sigma^6}  2 \sum_{\mathclap{\qquad(j_1,j_2,j_3,j_4,j_5,j_6) \in J_c}}
			\E[ Y_{j_1} Y_{j_2} Y_{j_3} Y_{j_4} Y_{j_5} Y_{j_6} ].
	\end{align*}
	Taking the root on both sides gives us two upper bounds for \eqref{scA3}.\\
	
	\noindent\textbf{Term \eqref{scA1}:} $\displaystyle\Big( \Var\Big(
				\sum_{\mathclap{\substack{j \in J \\ k \in N_j}}}
				X_j X_k
				f_{2,t} \Big( \sum_{m \in N_j} X_m \Big)
			\Big) \Big)^{1/2}$\\
	We want to apply the same calculation steps as above.
	However, this time the argument of the variance consists of a product of several random variables.
	Therefore, we can not simply repeat the calculation steps from above.
	The variance in \eqref{scA1} can be represented as
	\begin{align*}
		\sum_{\substack{ j_1 \in J, k_1 \in N_{j_1} \\ j_2 \in J, k_2 \in N_{j_2}}}
			\Cov\Big( X_{j_1} X_{k_1} f_{2,t}\Big( \sum_{m_1 \in N_{j_1}} X_{m_1} \Big)
					, X_{j_2} X_{k_2} f_{2,t}\Big( \sum_{m_2 \in N_{j_2}} X_{m_2} \Big) \Big)
	\end{align*}
	First, we split the sum, according to the dependency structure between $X_{j_1}$, $X_{k_1}$, $X_{j_2}$ and $X_{k_2}$.
	Denote by $A = \{ j_1 \in J, k_1 \in N_{j_1}, j_2 \in J, k_2 \in N_{j_2} \}$ and $B = \{ j_2 \in N_{j_1} \}$, $C= \{ j_2 \in N_{k_1} \}$,
	$D = \{ k_2 \in N_{j_1} \}$ and $E= \{ k_2 \in N_{k_1} \}$. We choose the splitting into five sets as
	$$
	A = (A \cap B) \cup (A \cap B^c \cap C) \cup (A \cap B^c \cap C^c \cap D) \cup  (A \cap B^c \cap C^c \cap D^c \cap E) \cup (A \cap B^c \cap C^c \cap D^c \cap E^c).
	$$
	In the summation for the first 4 sets, we observe that 
	these sums contain all cases, in which
	$X_{j_1} X_{k_1}$ and $X_{j_2} X_{k_2}$ are dependent.
	Using
	the definition of $f_{2,t}$,
	the sesquilinearity of the covariance
	and some calculations, which have been exemplified in \autoref{sc-EWs-connect},
	we get as an upper bound for each of those four sums 
	on the one hand
	\begin{align*}
		2 \cdot \Vert R_2 \Vert_\infty^2 \cdot \frac{D^5}{\sigma^5} \cdot \sum_{j \in J} \E[ \vert X_j \vert ]
		={}& \frac12 \frac{D^5}{\sigma^5} \sum_{j \in J} \E[ \vert X_j \vert ],
	\end{align*}
	and on the other hand 
	\begin{align*}
		\Vert R_2 \Vert_\infty^2 \frac{2^7}{\sigma^6} \sum_{(j_1,j_2,j_3,j_4,j_5,j_6) \in J_c}
			\E[ Y_{j_1} Y_{j_2} Y_{j_3} Y_{j_4} Y_{j_5} Y_{j_6} ] 
		= \frac{32}{\sigma^6} \sum_{(j_1,j_2,j_3,j_4,j_5,j_6) \in J_c}
			\E[ Y_{j_1} Y_{j_2} Y_{j_3} Y_{j_4} Y_{j_5} Y_{j_6} ].
	\end{align*}

	It remains to analyze term 
	\begin{align}
			\sum_{\mathclap{\substack{ j_1 \in J, k_1 \in N_{j_1} \\j_2 \in N_{j_1}^C \cap N_{k_1}^C \\ k_2 \in N_{j_2} \cap N_{j_1}^C \cap N_{k_1}^C}}}
			\Cov\Big( X_{j_1} X_{k_1} f_{2,t}\Big( \sum_{m_1 \in N_{j_1}} X_{m_1} \Big)
					, X_{j_2} X_{k_2} f_{2,t}\Big( \sum_{m_2 \in N_{j_2}} X_{m_2} \Big) \Big)
			\label{S1A5}
	\end{align}
	which contains all cases, in which
	$X_{j_1} X_{k_1}$ and $X_{j_2} X_{k_2}$
	are independent.
	Obviously, if $\{X_{m_1}\}_{m_1 \in N_{j_1}}$ and $\{X_{m_2}\}_{m_2 \in N_{j_2}}$ are independent too,
	the Covariance is $0$.
	So now, suppose that $X_{j_1} X_{k_1}$ and $X_{j_2} X_{k_2}$ are independent,
	but $\sum_{m_1 \in N_{j_1}} X_{m_1}$ and $\sum_{m_2 \in N_{j_2}} X_{m_2}$ do depend on each other.\\
	We will now use the sesquilinearity of the covariance
	to split the sum $\sum_{m_1 \in N_{j_1}} X_{m_1}$ according to the dependence on $X_{j_2}$ and $X_{k_2}$ and vice versa.
	\begin{align}
		&\phantom{+} \Cov\Big(
				X_{j_1} X_{k_1} f_{2,t}\big( \sum_{m_1 \in N_{j_1}} X_{m_1} \big) ,
				X_{j_2} X_{k_2} f_{2,t}\big( \sum_{m_2 \in N_{j_2}} X_{m_2} \big)		\Big)
				\nonumber\\
		={}&\phantom{+} \Cov\Big(
				X_{j_1} X_{k_1} f_{2,t}\big( \qquad\sum_{\mathclap{m_1 \in N_{j_1} \cap N_{j_2}^C \cap N_{k_2}^C}} \; X_{m_1} \;\; \big) ,
				X_{j_2} X_{k_2} f_{2,t}\big( \qquad\sum_{\mathclap{m_2 \in N_{j_2} \cap N_{j_1}^C \cap N_{k_1}^C}} \; X_{m_2} \;\; \big)
				\Big)
				\label{S1A6}\\
		&+ \Cov\Big(
				X_{j_1} X_{k_1}	\Big(	f_{2,t}\big( \sum_{m_1 \in N_{j_1}} X_{m_1} \big)
							-f_{2,t}\big( \qquad\sum_{\mathclap{m_1 \in N_{j_1} \cap N_{j_2}^C \cap N_{k_2}^C}} \; X_{m_1} \;\; \big) \Big) ,
				X_{j_2} X_{k_2} f_{2,t}\big( \qquad\sum_{\mathclap{m_2 \in N_{j_2} \cap N_{j_1}^C \cap N_{k_1}^C}} \; X_{m_2} \;\; \big)
				\Big)
				\label{S1A7}\\
		&+ \Cov\Big(
				X_{j_1} X_{k_1} f_{2,t}\big( \sum_{m_1 \in N_{j_1}} X_{m_1} \big) ,
				X_{j_2} X_{k_2} 	\Big(	f_{2,t}\big( \sum_{m_2 \in N_{j_2}} X_{m_2} \big)
							-f_{2,t}\big( \qquad\sum_{\mathclap{m_2 \in N_{j_2} \cap N_{j_1}^C \cap N_{k_1}^C}} \; X_{m_2} \;\; \big) \Big)
							\Big)
				\label{S1A8}
	\end{align}
	Using
	the definition of $f_{2,t}$,
	some calculations, which have been exemplified in \autoref{sc-EWs-connect},
	and
	the mean value theorem,
	we get the following upper bound for \eqref{S1A7}:
	\begin{align*}
		&\Big\vert \Cov\Big(
				X_{j_1} X_{k_1}	\Big(	f_{2,t}\big( \sum_{m_1 \in N_{j_1}} X_{m_1} \big)
							-f_{2,t}\big( \qquad\sum_{\mathclap{m_1 \in N_{j_1} \cap N_{j_2}^C \cap N_{k_2}^C}} \; X_{m_1} \;\; \big) \Big) ,
				X_{j_2} X_{k_2} f_{2,t}\big( \qquad\sum_{\mathclap{m_2 \in N_{j_2} \cap N_{j_1}^C \cap N_{k_1}^C}} \; X_{m_2} \;\; \big)
				\Big)
			\Big\vert \\
		\leq{}&\E\Big[
				\vert X_{j_1} X_{k_1} \vert
					\Vert f_{2,t}^\prime \Vert_\infty \qquad\sum_{\mathclap{m_1 \in N_{j_1} \cap (N_{j_2} \cup N_{k_2})}} \; \vert X_{m_1} \vert \;\;
				\cdot
				\vert X_{j_2} X_{k_2} \vert
					\Vert R_2 \Vert_\infty \qquad\sum_{\mathclap{m_2 \in N_{j_2} \cap (N_{j_1} \cup N_{k_1})^C}} \; \vert X_{m_2} \vert \;\;
			\Big]
			\\
		& + \E\Big[
				\vert X_{j_1} X_{k_1} \vert
					\Vert f_{2,t}^\prime \Vert_\infty \qquad\sum_{\mathclap{m_1 \in N_{j_1} \cap (N_{j_2} \cup N_{k_2})}} \; \vert X_{m_1} \vert \;\;
			\Big] \cdot \E\Big[
				\vert X_{j_2} X_{k_2} \vert
					\Vert R_2 \Vert_\infty \qquad\sum_{\mathclap{m_2 \in N_{j_2} \cap (N_{j_1} \cup N_{k_1})^C}} \; \vert X_{m_2} \vert \;\;
			\Big]
			\\
		\leq{}&
			\sum_{\substack{m_1 \in N_{j_1} \cap (N_{j_2} \cup N_{k_2}) \\ m_2 \in N_{j_2} \cap (N_{j_1} \cup N_{k_1})^C}}
			\min \Big\{
				2 \cdot \frac1{\sigma^5} \cdot \E[ \vert X_{j_1} \vert ] \; , \;
				2 \cdot \big(\frac2\sigma\big)^6 \cdot \E[ Y_{j_1} Y_{k_1} Y_{m_1} Y_{j_2} Y_{k_2} Y_{m_2} ]
			\Big\}.
	\end{align*}
	When we will merge all upper bounds for \eqref{S1A6}, \eqref{S1A7}, \eqref{S1A8} to derive upper bounds for \eqref{S1A5},
	we will have to sum these bounds up regarding some tupels $(j_1,k_1,j_2,k_2)$.
	Notice that he set 
	$$
	A := \Big\{ k_1 \in N_{j_1}, m_1 \in N_{j_1} \cap (N_{j_2} \cup N_{k_2}),  j_2 \in N_{j_1}^C \cap N_{k_1}^C,
					k_2 \in N_{j_2} \cap N_{j_1}^C \cap N_{k_1}^C, m_2 \in N_{j_2} \cap (N_{j_1} \cup N_{k_1})^C \Big\}
	$$
	has at most $ 2 D^5$ elements. Hence, \eqref{S1A7} will contribute to the upper bounds for \eqref{S1A5} with the quantity
	on the one hand
	\begin{align*}
       4 \cdot \frac{D^5}{\sigma^5} \cdot \sum_{j \in J} \E[ \vert X_j \vert ]
	\end{align*}
	and on the other hand
	\begin{align*}
		\frac{128}{\sigma^6} 
			\sum_{(j_1,j_2,j_3,j_4,j_5,j_6) \in J_c}
			\E[ Y_{j_1} Y_{j_2} Y_{j_3} Y_{j_4} Y_{j_5} Y_{j_6} ].
	\end{align*}
	\eqref{S1A8} can be bounded analogously, so it remains to analyze \eqref{S1A6}.\\
	We still want to use the same strategy as we did in the proof concerning \eqref{scA3}.
	Therefore, we need to draw the factors $X_{j_1} X_{k_1}$ and $X_{j_2} X_{k_2}$ out of the covariances.
	We will achieve this by means of the conditional covariance.
	Let $\mathcal{F} := \sigma( I_a \;\vert\; a \in j_1 \cup k_1 \cup j_2 \cup k_2 )$ be the $\sigma$-algebra
	generated by the indicators of the edges of the graphs $j_1$, $k_1$, $j_2$, $k_2$.
	According to the law of total covariance, we get
	\begin{align*}
		&\Cov\Big(
				X_{j_1} X_{k_1} f_{2,t}\big( \qquad\sum_{\mathclap{m_1 \in N_{j_1} \cap N_{j_2}^C \cap N_{k_2}^C}} \; X_{m_1} \;\; \big) ,
				X_{j_2} X_{k_2} f_{2,t}\big( \qquad\sum_{\mathclap{m_2 \in N_{j_2} \cap N_{j_1}^C \cap N_{k_1}^C}} \; X_{m_2} \;\; \big)
				\Big)\\
		={}& \E \Bigg( \Cov\Big(
				X_{j_1} X_{k_1} f_{2,t}\big( \qquad\sum_{\mathclap{m_1 \in N_{j_1} \cap N_{j_2}^C \cap N_{k_2}^C}} \; X_{m_1} \;\; \big) ,
				X_{j_2} X_{k_2} f_{2,t}\big( \qquad\sum_{\mathclap{m_2 \in N_{j_2} \cap N_{j_1}^C \cap N_{k_1}^C}} \; X_{m_2} \;\; \big)
					\;\Big\vert\; \mathcal{F} \Big) \Bigg)\\
		&+ \Cov\Bigg(
				\E\Big(
				X_{j_1} X_{k_1} f_{2,t}\big( \qquad\sum_{\mathclap{m_1 \in N_{j_1} \cap N_{j_2}^C \cap N_{k_2}^C}} \; X_{m_1} \;\; \big)
					\;\Big\vert\; \mathcal{F} \Big) ,
				\E\Big( 
				X_{j_2} X_{k_2} f_{2,t}\big( \qquad\sum_{\mathclap{m_2 \in N_{j_2} \cap N_{j_1}^C \cap N_{k_1}^C}} \; X_{m_2} \;\; \big)
					\;\Big\vert\; \mathcal{F} \Big)		\Bigg).
	\end{align*}
	However, we have already removed all tupels of indices, that could lead to difficulties,
	so that the conditional expectations above are always independent.
	Their covariance hence is $0$.
	It remains to analyze the expectation of the conditional covariance.
	Using that $X_{j_1}, X_{k_1}, X_{j_2}, X_{k_2}$ are $\mathcal{F}$-measurable, we get
	\begin{align*}
		&\Cov\Big(
				X_{j_1} X_{k_1} f_{2,t}\big( \qquad\sum_{\mathclap{m_1 \in N_{j_1} \cap N_{j_2}^C \cap N_{k_2}^C}} \; X_{m_1} \;\; \big) ,
				X_{j_2} X_{k_2} f_{2,t}\big( \qquad\sum_{\mathclap{m_2 \in N_{j_2} \cap N_{j_1}^C \cap N_{k_1}^C}} \; X_{m_2} \;\; \big)
				\Big)\\
		={}& \E \Bigg( X_{j_1} X_{k_1} \cdot X_{j_2} X_{k_2} \cdot \Cov\Big(
				f_{2,t}\big( \qquad\sum_{\mathclap{m_1 \in N_{j_1} \cap N_{j_2}^C \cap N_{k_2}^C}} \; X_{m_1} \;\; \big) ,
				f_{2,t}\big( \qquad\sum_{\mathclap{m_2 \in N_{j_2} \cap N_{j_1}^C \cap N_{k_1}^C}} \; X_{m_2} \;\; \big)
				\;\Big\vert\; \mathcal{F} \Big) \Bigg).
	\end{align*}
	Recall that sums like
	$\sum_{m_1 \in N_{j_1} \cap N_{j_2}^C \cap N_{k_2}^C} X_{m_1}$
	and
	$\sum_{m_2 \in N_{j_2} \cap N_{j_1}^C \cap N_{k_1}^C} X_{m_2}$
	are associated.
	Now, for the next step it is crucial to notice
	that those sums of random variables
	given \linebreak$\{ I_a \}_{a \in j_1 \cup k_1 \cup j_2 \cup k_2}$
	are still associated.
	Thus, we may again apply Lemma~3 from \cite{Ne80}.
	\begin{align*}
		&\Big\vert \Cov\Big(
				X_{j_1} X_{k_1} f_{2,t}\big( \qquad\sum_{\mathclap{m_1 \in N_{j_1} \cap N_{j_2}^C \cap N_{k_2}^C}} \; X_{m_1} \;\; \big) ,
				X_{j_2} X_{k_2} f_{2,t}\big( \qquad\sum_{\mathclap{m_2 \in N_{j_2} \cap N_{j_1}^C \cap N_{k_1}^C}} \; X_{m_2} \;\; \big)
				\Big)\Big\vert\\
		&\leq \Vert f^\prime_{2,t} \Vert_\infty^2 \cdot
			\sum_{\substack{m_1 \in N_{j_1} \cap N_{j_2}^C \cap N_{k_2}^C \\ m_2 \in N_{j_2} \cap N_{j_1}^C \cap N_{k_1}^C}}
			\E \Big( \vert X_{j_1} X_{k_1} \cdot X_{j_2} X_{k_2} \vert \cdot
			\big\vert \Cov \big( X_{m_1}, X_{m_1}  \;\big\vert\; \mathcal{F} \big) \big\vert \Big)\\
		&= \Vert f^\prime_{2,t} \Vert_\infty^2 \cdot
			\sum_{\substack{m_1 \in N_{j_1} \cap N_{j_2}^C \cap N_{k_2}^C \phantom{\cap N_{m_1}} \\
						m_2 \in N_{j_2} \cap N_{j_1}^C \cap N_{k_1}^C \cap N_{m_1}}}
			\E \Big( \vert X_{j_1} X_{k_1} \cdot X_{j_2} X_{k_2} \vert \cdot
			\big\vert \Cov \big( X_{m_1}, X_{m_1}  \;\big\vert\; \mathcal{F} \big) \big\vert \Big).
	\end{align*}
	It is possible to calculate that $\Vert f^\prime_{2,t} \Vert_\infty = \frac12$.
	Using again
	some calculations, which have been exemplified in \autoref{sc-EWs-connect},
	we get the following upper bounds for \eqref{S1A6}.
	On the one hand
	\begin{align*}
		&\big(\frac12\big)^2 \cdot \frac2{\sigma^5} \cdot 2 \cdot
			\sum_{\substack{m_1 \in N_{j_1} \cap N_{j_2}^C \cap N_{k_2}^C \phantom{\cap N_{m_1}} \\
						m_2 \in N_{j_2} \cap N_{j_1}^C \cap N_{k_1}^C \cap N_{m_1}}}
			\E[ \vert X_{j_1} \vert ]
	\end{align*}
	and on the other hand
	\begin{align*}
		&\big(\frac12\big)^2 \cdot \big(\frac2\sigma\big)^6 \cdot 2 \cdot
			\sum_{\substack{m_1 \in N_{j_1} \cap N_{j_2}^C \cap N_{k_2}^C \phantom{\cap N_{m_1}} \\
						m_2 \in N_{j_2} \cap N_{j_1}^C \cap N_{k_1}^C \cap N_{m_1}}}
			\E[ Y_{j_1} Y_{k_1} Y_{m_1} Y_{j_2} Y_{k_2} Y_{m_2} ]
	\end{align*}
	Note again that, when we will merge all upper bounds for \eqref{S1A6}, \eqref{S1A7}, \eqref{S1A8} to construct upper bounds for \eqref{S1A5},
	we will have to sum these bounds up regarding some tupels $(j_1,k_1,j_2,k_2)$.
	After this summation, \eqref{S1A6} will contribute to the upper bounds for \eqref{S1A5} with the following quantities.
	On the one hand
	\begin{align*}
		\frac{D^5}{\sigma^5} 
			\sum_{j \in J}
			\E[ \vert X_{j_1} \vert ]
	\end{align*}
	and on the other hand
	\begin{align*}
		\frac{32}{\sigma^6} 
			\sum_{(j_1,j_2,j_3,j_4,j_5,j_6) \in J_c}
			\E[ Y_{j_1} Y_{k_1} Y_{m_1} Y_{j_2} Y_{k_2} Y_{m_2} ].
	\end{align*}
	
	Finally, we can join all upper bounds from this part of the proof to construct upper bounds for \eqref{scA1}.
	\begin{align*}
		\Var\Big(
				\sum_{\mathclap{\substack{j \in J \\ k \in N_j}}}
				X_j X_k
				f_{2,t} \Big( \sum_{m \in N_j} X_m \Big)
			\Big),
	\end{align*}
	which is the square of \eqref{scA1},
	is bounded by
	\begin{align*}
		11 \frac{D^5}{\sigma^5} \sum_{j \in J} \E[ \vert X_{j_1} \vert ]
	\,\,\,\,  \text{and by} \, \, \,\,
		\frac{320}{\sigma^6} \sum_{(j_1,j_2,j_3,j_4,j_5,j_6) \in J_c} \E[ Y_{j_1} Y_{k_1} Y_{l_1} Y_{j_2} Y_{k_2} Y_{l_2} ].
	\end{align*}
	Now it only remains to bound \eqref{scA2}.\\

	\noindent\textbf{Term \eqref{scA2}:} $\displaystyle\Big( \Var\Big(
				\sum_{\mathclap{\substack{j \in J \\ k \in N_j}}}
				X_j X_k
				f_{1,t} \Big( \sum_{m \in N_j \cup N_k} X_m \Big)
			\Big) \Big)^{1/2}$\\
	The variance can be bounded analogously to the previous one.
	Therefore, we will only point out what has to be changed in the calculations above.
	Instead of $f_{2,t}$ we now have to use $f_{1,t}$.
	Since $\Vert f^\prime_{1,t} \Vert_\infty^2 = 4 \cdot \Vert f^\prime_{2,t} \Vert_\infty^2$
	we here get $4$ as an extra factor in the bound compared to the bound from the previous part of the proof.
	In the argument of $f_{1,t}$ we now have more summands than in the argument of $f_{2,t}$, but at most double as many.
	Hence, we get again an extra factor $2^2=4$.
	We conclude, that
	\begin{align*}
		\Var\Big(
				\sum_{\mathclap{\substack{j \in J \\ k \in N_j}}}
				X_j X_k
				f_{1,t} \Big( \sum_{m \in N_j \cup N_k} X_m \Big)
			\Big)
	\end{align*}
	is bounded by
	\begin{align*}
		176 \frac{D^5}{\sigma^5} \sum_{j \in J} \E[ \vert X_{j_1} \vert ]
\,\,\,\, \text{and by} \,\,\,\,
		\frac{5120}{\sigma^6} \sum_{(j_1,j_2,j_3,j_4,j_5,j_6) \in J_c} \E[ Y_{j_1} Y_{k_1} Y_{l_1} Y_{j_2} Y_{k_2} Y_{l_2} ].
	\end{align*}\mbox{}\\
	
	\noindent\textbf{Summary:}
	We now have bounded \eqref{scA1}, \eqref{scA2}, \eqref{scA3} separately.
	Merging all previous results implies after a small calculation the following bounds.
	On the one hand
	\begin{align*}
		\vert \Cov( H_t  , e^{-itW}) \vert
		\leq 20\frac{D^\frac52}{\sigma^\frac52} \cdot \Big( \sum_{j \in J} \E[ \vert X_j \vert ] \Big)^{1/2}
	\end{align*}
	and on the other hand
	\begin{align*}
		\vert \Cov( H_t  , e^{-itW}) \vert
		\leq \frac{113}{\sigma^3} 
			\Big( \sum_{(j_1,j_2,j_3,j_4,j_5,j_6) \in J_c	} \E[ Y_{j_1} Y_{j_2} Y_{j_3} Y_{j_4} Y_{j_5} Y_{j_6} ] \Big)^{1/2}.
	\end{align*}
	Therefore, the proof is complete.
\end{proof}

Now, as \autoref{sc-SchrankenH} and \autoref{sc-SchrankenC} have been proved,
we continue by examining the quantities
$\displaystyle\sum_{j \in J} \E[ \vert X_j \vert ]$,
$\displaystyle\sum_{\substack{j \in J \\ k, l \in N_j}} \E[ Y_j Y_k Y_l ]$
and
$\displaystyle\Big( \sum_{(j_1,j_2,j_3,j_4,j_5,j_6) \in J_c} \E[ Y_{j_1} Y_{j_2} Y_{j_3} Y_{j_4} Y_{j_5} Y_{j_6} ] \Big)^{1/2}$,
which appear as upper bounds in those two lemmas.

\begin{lem}\label{Sechser-EW}
	It holds that
	\begin{align}
		\sum_{j \in J} \E[ \vert X_j \vert ]
		&\leq \frac2\sigma \cdot n^{v_\G} \cdot e_\G \cdot (1-p), \label{end1}\\
		\sum_{\substack{j \in J \\ k \in N_j \\ l \in N_j}} \E[ Y_j Y_k Y_l ]
		&\leq C_\G
			\Psi_\G^3
			\Big( \sum_{h \subset \G} \Psi_h^{-1} \Big)^2, \label{end2}\\
		\sum_{(j_1,j_2,j_3,j_4,j_5,j_6) \in J_c} \E[ Y_{j_1} Y_{j_2} Y_{j_3} Y_{j_4} Y_{j_5} Y_{j_6} ]
		&\leq C_\G
			\Psi_\G^6
			\Big( \sum_{h \subset \G} \Psi_h^{-1} \Big)^5. \label{end3}
	\end{align}
\end{lem}
\begin{proof}
	The cardinality of $J$ can be computed combinatorically:
	\begin{align*}
		\vert J \vert \leq \frac{n!}{(n-v_\G)!} \leq n^{v_\G}.
	\end{align*}
	In addition, we find the following bound for the expectation of $\vert X_j \vert$:
	\begin{align*}
		\E[ \vert X_j \vert ] & = \frac1\sigma \E [ \vert (1-Y_j) - \E[1-Y_j] \vert ] \\
		&\leq \frac2\sigma  \E [ 1-Y_j ] = \frac2\sigma (1-p^{e_\G}) \\
		&= \frac2\sigma (1-p) \cdot \sum_{a=0}^{e_\G-1} p^a 
		\leq \frac2\sigma (1-p) \cdot e_\G
	\end{align*}
	These inequalities together imply \eqref{end1}.\\
	
	The proofs of \eqref{end2} and \eqref{end3} resemble each other.
	Hence, we will only show the proof of \eqref{end3}.
	Our strategy is based on a proof from \cite{BKR89} and generalizes it.\\
	First, we want to illustrate the main idea:
	Let $j \in J$ be a fixed graph.
	We now want to calculate an upper bound for the number of its neighbors,
	which is the cardinality of $N_j$.
	For any graph $k \in J$, it holds, that $k$ is a neighbor of $j$ ($k \in N_j$),
	if and only if the intersection of $k$ und $j$ is a graph with at least one edge.
	Therefore, we will fix a graph $h \subset j$ and then count,
	how many graphs $k \in J$ exist so that the intersection of $j$ and $k$ is equal to $h$ ($j \cap k = h$).
	In this case, obviously, $h$ is a subgraph of $k$.
	So, $k$ consists of the $v_h$ vertices of $h$ and $v_\G - v_h$ other vertices.
	The latter can be arbitrarily chosen among the $n - v_h$ spare vertices.
	This gives us $\binom{n-v_h}{v_\G-v_h}$ possible sets of vertices for $k$.
	To derive the number of neighbors $\vert N_j \vert$,
	we need to multiply this binomial coefficient with
	the number of subgraphs, which exist on $v_\G$ given vertices and include the subgraph $h$.
	This constant does only depend on $\G$.
	This strategy will be executed repeatedly.
	During the following calculation, we will use the notation ``$\underset{\G}{\subset}$''.
	This will denote, that the left side is a subgraph of the right side, where the left side in addition is isomorphic to a subgraph of $\G$.
	We finally find:
	\begin{align*}
		&\sum_{(j_1,j_2,j_3,j_4,j_5,j_6) \in J_c} \E[ Y_{j_1} Y_{j_2} Y_{j_3} Y_{j_4} Y_{j_5} Y_{j_6} ]\\[1em]
		&\leq
			\sum_{j_1 \in J}
			\sum_{h_1 \underset{\G}{\subset} j_1}
			\sum_{\substack{j_2 \in J \\ j_1 \cap j_2 = h_1}}
			\sum_{h_2 \underset{\G}{\subset} j_1 \cup j_2}
			\sum_{\substack{j_3 \in J \\ (j_1 \cup j_2) \cap j_3 = h_2}}
			\sum_{h_3 \underset{\G}{\subset} j_1 \cup j_2 \cup j_3}
			\sum_{\substack{j_4 \in J \\ (j_1 \cup j_2 \cup j_3) \cap j_4 = h_3}}
			\sum_{h_4 \underset{\G}{\subset} j_1 \cup j_2 \cup j_3 \cup j_4}\\
			&\hspace{1cm}
			\sum_{\substack{j_5 \in J \\ (j_1 \cup j_2 \cup j_3 \cup j_4) \cap j_5 = h_4}}
			\sum_{h_5 \underset{\G}{\subset} j_1 \cup j_2 \cup j_3 \cup j_4 \cup j_5}
			\sum_{\substack{j_6 \in J \\ (j_1 \cup j_2 \cup j_3 \cup j_4 \cup j_5) \cap j_6 = h_5}}
			p^{6e_\G - e_{h_1} - e_{h_2} - e_{h_3} - e_{h_4} - e_{h_5}}\\[1em]
		&\leq
			C_\G
			\binom{n}{v_\G}
			\sum_{h_1 \subset\G}
			\binom{n-v_{h_1}}{v_\G-v_{h_1}}
			\sum_{h_2 \subset\G}
			\binom{n-v_{h_2}}{v_\G-v_{h_2}}
			\sum_{h_3 \subset\G}
			\binom{n-v_{h_3}}{v_\G-v_{h_3}}\\
			&\hspace{3cm}
			\sum_{h_4 \subset\G}
			\binom{n-v_{h_4}}{v_\G-v_{h_4}}
			\sum_{h_5 \subset\G}
			\binom{n-v_{h_5}}{v_\G-v_{h_5}}
			p^{6e_\G - e_{h_1} - e_{h_2} - e_{h_3} - e_{h_4} - e_{h_5}}\displaybreak\\[1em]
		&\leq
			C_\G
			\sum_{h_1,h_2,h_3,h_4,h_5 \subset\G}
			\binom{n}{v_\G}
			\binom{n-v_{h_1}}{v_\G-v_{h_1}}
			\binom{n-v_{h_2}}{v_\G-v_{h_2}}
			\binom{n-v_{h_3}}{v_\G-v_{h_3}}\\
			&\hspace{3cm}
			\binom{n-v_{h_4}}{v_\G-v_{h_4}}
			\binom{n-v_{h_5}}{v_\G-v_{h_5}}
			p^{6e_\G - e_{h_1} - e_{h_2} - e_{h_3} - e_{h_4} - e_{h_5}}\\[1em]
		&\leq
			C_\G
			\sum_{h_1,h_2,h_3,h_4,h_5 \subset\G}
			n^{6v_\G - v_{h_1} - v_{h_2} - v_{h_3} - v_{h_4} - v_{h_5}}
			p^{6e_\G - e_{h_1} - e_{h_2} - e_{h_3} - e_{h_4} - e_{h_5}}\\
		&=
			C_\G
			(n^{v_\G}p^{e_\G})^6
			\Big( \sum_{h \subset \G} n^{-v_h}p^{-e_h} \Big)^5\\
		&=
			C_\G
			\Psi_\G^6
			\Big( \sum_{h \subset \G} \Psi_h^{-1} \Big)^5. \qedhere
	\end{align*}
\end{proof}

Before we can start the final part of our main theorem's proof,
we need a lower bound for the variance $\sigma^2$ of our random variable $W$.
The following \autoref{satz:sigma} is a direct consequence of Lemma~3.5 from \cite{JLR00}.
Therefore, we will omit the proof.

\begin{cor}\label{satz:sigma}
	There are constants $\{ c_H \}_{H \subset \G}$, only depending on $\G$, so that
	\begin{align*}
		(1-p) \cdot c_{\tilde H} \cdot \frac{\Psi_G^2}{\Psi_{\tilde H}}
		\leq \sigma^2,
	\end{align*}
	where $\tilde H \subset \G$ is an arbitrarily chosen subgraph.
\end{cor}
\noindent Note that the inequality in \autoref{satz:sigma} especially holds for $\Psi_{\tilde H} = \Psi$.\\

Now, we are ready to put things together to proof Theorem 1.1.

\begin{proof}[Proof (of \autoref{maintheorem})]
	The combination of \autoref{sc-SchrankenH}, \autoref{sc-SchrankenC} and \autoref{Sechser-EW} implies, that
	\begin{align*}
		\vert \E[H_t] \vert
			&\leq C_\G \cdot \frac{D^2 \cdot n^{v_\G} \cdot (1-p)}{\sigma^3},\\
		\vert \E[H_t] \vert
			&\leq C_\G \cdot \frac{\Psi_\G^3 \Big( \sum_{h \subset \G} \Psi_h^{-1} \Big)^2}{\sigma^3},\\
		\vert \Cov(H_t , e^{-itW}) \vert
			&\leq C_\G \cdot \frac{D^\frac52 \cdot n^\frac{v_\G}2 \cdot (1-p)^\frac12}{\sigma^3},\\
		\vert \Cov(H_t , e^{-itW}) \vert
			&\leq C_\G \cdot \frac{\Psi_\G^3 \Big( \sum_{h \subset \G} \Psi_h^{-1} \Big)^\frac52}{\sigma^3}.
	\end{align*}
	Plugging these inequalities into \autoref{theo:STMresult} leads to
	\begin{align*}
		\Kol(\L(W), \Normal(0,1))
			\leq{}&
			C_\G \cdot \Bigg( \frac{D^2 \cdot n^{v_\G} \cdot (1-p)}{\sigma^3}
			+ \frac{D^\frac12}{ n^\frac{v_\G}2 \cdot (1-p)^\frac12} \Bigg),\\
		\Kol(\L(W), \Normal(0,1))
			\leq{}&
			C_\G \cdot \Bigg( \frac{\Psi_\G^3 \Big( \sum_{h \subset \G} \Psi_h^{-1} \Big)^2}{\sigma^3}
			+ \Big( \sum_{h \subset \G} \Psi_h^{-1} \Big)^\frac12 \Bigg).
	\end{align*}
	It is easy to prove combinatorically, that
	\begin{align*}
		D &\leq C_\G \cdot \binom{n}{v_\G-2} \leq C_\G \cdot n^{v_\G -2}.
	\end{align*}
	It also holds, that
	\begin{align*}
		\sum_{h \subset \G} \Psi_h^{-1} &\leq C_\G \cdot \Psi^{-1}.
	\end{align*}
	In addition, \autoref{satz:sigma} implies, that
	\begin{align*}
		\frac1{\sigma^3}
		\leq C_\G \cdot
		\frac1{(1-p)^\frac32} \cdot \frac{\Psi^\frac32}{\Psi_\G^3}.
	\end{align*}
	These inequalities, together with the definition of $\Psi_\G = n^{v_\G} p^{e_\G}$
	and with the already known inequality $\Psi \leq n^2 \cdot p$, lead to
	\begin{align*}
		\Kol(\L(W), \Normal(0,1))
			\leq{}&
			C_\G \cdot \Bigg( \frac{n^{2v_\G-4} \cdot n^{v_\G} \cdot (1-p) \cdot n^3 \cdot p^\frac32}{(1-p)^\frac32 \cdot n^{3v_\G} p^{3e_\G}}
			+ \frac{n^{\frac{v_\G}2-1}}{ n^\frac{v_\G}2 \cdot (1-p)^\frac12} \Bigg),\\
		\Kol(\L(W), \Normal(0,1))
			\leq{}&
			C_\G \cdot \Bigg( \frac{1}{\Psi^\frac12 (1-p)^\frac32}
			+ \frac1{\Psi^\frac12} \Bigg).
	\end{align*}
	Thus,
	\begin{align*}
		\Kol(\L(W), \Normal(0,1))
			\leq{}&
			C_\G \cdot \frac{1}{ n \cdot \sqrt{1-p}} \cdot
			\Bigg( 1 + \frac{1}{p^{3(e_\G-\frac12)}} \Bigg),\\
		\Kol(\L(W), \Normal(0,1))
			\leq{}&
			C_\G \cdot \frac1{\sqrt{\Psi}} \cdot
			\Bigg( 1 + \frac{1}{(1-p)^\frac32} \Bigg).
	\end{align*}
	Now, let $p_0 \in (0,1)$.
	We can eventually conclude:
	\begin{align*}
		\Kol(\L(W), \Normal(0,1))
		\leq{}& C_{\G,p_0} \cdot
		\begin{dcases}
			\frac1{n\sqrt{1-p}}
			& \text{ if }p > p_0\\
			\frac1{\sqrt{\Psi}}
			& \text{ if }p \leq p_0
		\end{dcases}
	\end{align*}
	This proves the main theorem (\autoref{maintheorem}).
\end{proof}

\begin{rem}
	At last, we want to note that the main theorem still holds,
	if $\G$ has isolated vertices, as long as it has at least one edge:
	Let $a \in \N$ be the number of isolated vertices of $\G$.
	And let $\tilde\G \subset \G$ contain all edges of $\G$, but no isolated vertex.
	Then, the main theorem holds for $\tilde\G$.
	However, every copy of $\tilde\G$ in $G(n,p)$ can be completed to a copy of $\G$ in $\frac{(n-v_{\tilde\G})!}{(n-v_{\tilde\G}-a)!}$ different ways.
	Therefore, the standardized number of subgraphs of $G(n,p)$, which are isomorphic to $\tilde\G$,
	is identical to the standardized number of subgraphs of $G(n,p)$, which are isomorphic to $\G$.
\end{rem}



\begin{thebibliography}{18}

\bibitem[Arras et~al.(2017)Arras, Mijoule, Poly, and Swan]{AMPS17}
B.~Arras, G.~Mijoule, G.~Poly, and Y.~Swan.
\newblock \emph{A new approach to the Stein-Tikhomirov method with applications to the second Wiener chaos and Dickman convergence}.
\newblock Version 2.
\newblock June 27, 2017.
\newblock \arxiv{1605.06819v2}{math.PR}.

\bibitem[Barbour et~al.(1989)Barbour, Karo\'nski, and Ruci\'nski]{BKR89}
A.~D. Barbour, M.~Karo\'nski, and A.~Ruci\'nski.
\newblock ``A central limit theorem for decomposable random variables with applications to random graphs''.
\newblock In: \emph{Journal of Combinatorial Theory, Series B} 47 (2 1989), pp. 125--145.
\newblock \doi{10.1016/0095-8956(89)90014-2}.

\bibitem[Chen(1978)]{Chen78}
L.~Chen.
\newblock ``Two central limit problems for dependent random variables''.
\newblock In: \emph{Z. Wahrsch. Verw. Gebiete}, 43 (1978), pp. 223--243.
\newblock \doi{10.1007/BF00536204}.

\bibitem[Esary et~al.(1967)Esary, Proschan, and Walkup]{EPW67}
J.~D. Esary, F.~Proschan, and D.~W. Walkup.
\newblock ``Association of random variables, with applications''.
\newblock In: \emph{Annals of Mathematical Statistics} 38 (1967), pp. 1466--1474
\newblock \doi{10.1214/aoms/1177698701}.

\bibitem[F\'eray et~al.(2017)F\'eray, M\'eliot, and Nikeghbali]{FMN17}
V.~F\'eray, P.-L. M\'eliot, and A.~Nikeghbali.
\newblock \emph{Mod-$\phi$ convergence, II: Estimates on the speed of convergence}.
\newblock Version 2.
\newblock Feb 20, 2018.
\newblock \arxiv{1705.10485v2}{math.PR}.

\bibitem[Gilbert(1959)]{Gi59}
E.~N. Gilbert.
\newblock ``Random graphs''.
\newblock In: \emph{Annals of Mathematical Statistics} 30 (4 1959), pp. 1141--1144.
\newblock \doi{10.1214/aoms/1177706098}.

\bibitem[Janson et~al.(2000)Janson, Luczak, and Ruci\'nski]{JLR00}
S.~Janson, T.~Luczak, and A.~Ruci\'nski.
\newblock \emph{Random Graphs}.
\newblock New York et al.: John Wiley \& Sons, Inc., 2000.
\newblock \isbn{9781118032718}.
\newblock \doi{10.1002/9781118032718}.

\bibitem[Krokowski et~al.(2017)Krokowski, Reichenbachs, and Thäle]{KRT17}
K.~Krokowski, A.~Reichenbachs, and C.~Thäle.
\newblock ``Discrete malliavin-stein method: Berry-Esseen bounds for random graphs and percolation''.
\newblock In: \emph{The Annals of Probability} 45 (2 2017), pp. 1071--1109.
\newblock \doi{10.1214/15-AOP1081}.

\bibitem[Lo\`eve(1977)]{Lo77}
M.~Lo\`eve.
\newblock \emph{Graduate texts in mathematics}
\newblock Vol.~45: \emph{Probability Theory I}.
\newblock 4th edition.
\newblock New York, Heidelberg, and Berlin: Springer, 1977.
\newblock \isbn{3-540-90210-4}.

\bibitem[McGinley and Sibson(1975)]{McGSib75}
W.~McGinley and R.~Sibson.
\newblock ``Dissociated random variables''.
\newblock In: \emph{Math. Proc. Cambridge Philos. Soc.} 77 (1975), pp. 185--188.

\bibitem[Newman(1980)]{Ne80}
C.~M. Newman.
\newblock ``Normal fluctuations and the FKG inequalities''.
\newblock In: \emph{Communications in Mathematical Physics} 74 (2 1980), pp. 119--128.
\newblock \textsc{url:} \url{https://projecteuclid.org/euclid.cmp/1103907978}.

\bibitem[Privault and Serafin(2020)]{PS20}
N.~Privault and G.~Serafin.
\newblock ``Normal approximation for sums of discrete $U$-statistics -- application to Kolmogorov bounds in random subgraph counting''.
\newblock In: \emph{Bernoulli} 26 (1 2020), pp. 587--615.
\newblock \doi{10.3150/19-BEJ1141}.

\bibitem[Rai\v{c}(2003)]{R03}
M.~Rai\v{c}.
\newblock ``Normal approximation by Stein's method''.
\newblock In: \emph{Proceedings of the Seventh Young Statisticans Meeting} 21 (2003), pp. 71--97, 2003.

\bibitem[Röllin(2021)]{Ro17}
A.~Röllin.
\newblock ``Kolmogorov bounds for the normal approximation of the number of triangles in the Erdös-Rényi random graph''.
\newblock In: \emph{Probability in the Engineering and Informational Sciences} (2021), pp. 1–-27.
\newblock \doi{10.1017/S0269964821000061}.

\bibitem[Ross(2011)]{Ro11}
N.~Ross.
\newblock ``Fundamentals of Stein's method''.
\newblock In: \emph{Probability Surveys} 8 (2011), pp. 210--293.
\newblock \doi{10.1214/11-PS182}.

\bibitem[Ruci\'nski(1988)]{Ru88}
A.~Ruci\'nski.
\newblock ``When are small subgraphs of a random graph normally distributed?''.
\newblock In: \emph{Probability Theory and Related Fields} 78 (1988), pp. 1--10.
\newblock \doi{10.1007/BF00718031}.

\bibitem[Stein(1972)]{St72}
C.~Stein.
\newblock ``A bound for the error in the normal approximation to the distribution of a sum of dependent random variables''.
\newblock In: \emph{Proceedings of the Sixth Berkeley Symposium on Mathematical Statistics and Probability. Volume II: Probability Theory}.
\newblock Ed. by Le Cam, L., Neyman, J., and Scott, E.
\newblock Berkeley: University of California Press, 1972, pp. 583--602.
\newblock \textsc{url:} \url{https://projecteuclid.org/euclid.bsmsp/1200514239}.

\bibitem[Tikhomirov(1980)]{Ti80}
A.~N. Tikhomirov.
\newblock ``On the rate of convergence in the central limit theorem for weakly dependent random variables''.
\newblock In: \emph{Theory of Probability \& Its Applications} 25 (4 1980), pp. 790--809.
\newblock \doi{10.1137/1125092}.

\end{thebibliography}
\end{document}